\documentclass[12pt,a4paper]{article}
\usepackage[utf8]{inputenc}
\usepackage[english]{babel}
 \usepackage{mathrsfs}
\usepackage[T1]{fontenc}
\usepackage{amsmath}
\usepackage{amsfonts}
\usepackage{amssymb}
\usepackage{makeidx}
\usepackage{color}
\usepackage{amsthm}
\usepackage{setspace}
\let\oldbibliography\thebibliography
\renewcommand{\thebibliography}[1]{%
  \oldbibliography{#1}%
  \setlength{\itemsep}{0pt}%
}
\usepackage{tikz}
\usepackage{graphicx}
\usepackage[pdftex]{hyperref}
\usepackage[left=2cm,right=2cm,top=2cm,bottom=2cm]{geometry}
  \usepackage[all]{xy}
\newtheorem{thm}{Theorem}[section]
\newtheorem*{thma}{Theorem A}  
\newtheorem*{thmb}{Theorem B}

\newtheorem{lem}[thm]{Lemma}
\newcounter{counter_conj-lem}
\newtheorem{conj-lem}[thm]{Conjecture-Lemma}

\newtheorem{prop}[thm]{Proposition}
\newcounter{counter_conj-prop}
\newtheorem{conj-prop}[thm]{Conjecture-Proposition}

\newcounter{counter_conj-thm}
\newtheorem{conj-thm}[thm]{Conjecture-Theorem}

\newtheorem{cor}[thm]{Corollary}

\theoremstyle{remark}

\newtheorem{rem}[thm]{Remark}

\newtheorem{ex}[thm]{Example}

\newcommand{\black}{\color{black}}

\DeclareMathOperator{\C}{\mathbb{C}}

\DeclareMathOperator{\Z}{ \mathbb{Z} }

\author{Nguyen-Bac Dang, Thorsten Herrig\footnote{The second author was supported by DFG grant HE 8393/1-1.
}}

\title{Dynamical degrees of automorphisms on abelian varieties}

\begin{document}

\maketitle

\abstract{For any given Salem number, we construct  an automorphism on a simple abelian variety whose first dynamical degree is the square of the Salem number. Our construction works for both simple abelian varieties with totally indefinite quaternion multiplication and for simple abelian varieties of the second kind.
We then give a complete classification of the dynamical degree sequences for abelian varieties of dimension at most four and obtain an ergodic result for sequences of pullbacks of forms.
 }

\tableofcontents

\section*{Introduction}
\addcontentsline{toc}{section}{Introduction}

Let $X$ be a complex abelian variety of dimension $g$ and $f\colon X \longrightarrow X$ an endomorphism. Such a pair $(X,f)$ can be viewed as a dynamical system on the underlying (compact) complex torus or as an integral point inside the associated endomorphism algebra.
In this paper, we explore the connection between the underlying dynamical system induced by the map $f$ and the algebro-geometric and number-theoretic viewpoint.       
\medskip

Viewed as a dynamical system on the metric space $X$, the topological entropy $h_{\rm top}(f)$ of $f$ is a fundamental invariant that measures quantitatively how the orbits of $f$ "separate". 
By Gromov-Yomdin's theorem \cite{gromov}, this quantity can be expressed as
\begin{equation*}
 h_{\rm top}(f) = \max_{k \leqslant g} \log || f^* \colon{\rm H}^{k,k}(X) \longrightarrow {\rm H}^{k,k}(X) ||,
\end{equation*}
where $\rm H^{k,k}(X)$ denotes the $(k,k)$-Dolbeaut cohomology of $X$.
The spectral radius of $f^* $ on each $\rm H^{k,k}(X)$ provides a more refined dynamical invariant, the $k$-th dynamical degree $\lambda_k(f)$ of $f$.
These quantities measure the volume growth of the preimages by $f$ of complex subvarieties of codimension $k$ and $k \mapsto \lambda_k(f)$ forms a log concave sequence by Khovanskii-Teissier inequalities (see \cite[Example 1.6.4]{lazarsfeld_positivity_1}).
\smallskip

The dynamical degrees are Pisot or Salem numbers when $f$ is a surface automorphism or a birational surface map of positive entropy \cite{diller_favre,blanc_cantat}. 
Conversely, automorphisms with Salem entropies or Salem dynamical degrees were realized in the case of K3 surfaces \cite{McMullen11,McMullen02,oguiso}, Enriques surfaces \cite{MOR18,oguiso_yu_20}, two-dimensional complex tori \cite{Reschke12}, abelian surfaces \cite{Reschke17} (see also \cite{blanc_cantat} for Cremona transformations) and on certain  abelian varieties and Calabi-Yau varieties \cite{oguiso_pisot}. 
In this paper, we adopt the convention in which a Salem polynomial has at least one root of absolute value one (cf. \cite{bdgps}).

We are  interested in the dynamical spectra of a given automorphism $f$ on an abelian variety of arbitrary dimension $g$, namely on the sequences of the form:
\begin{equation*}
\Lambda(f):=(\lambda_1(f), \lambda_2(f), \ldots, \lambda_{g-1}(f)).
\end{equation*}
In this particular situation, the knowledge of the spectra of a given endomorphism allows one to deduce if there are preserved fibrations on threefolds \cite{truong_oguiso_salem}.
 Moreover, the precise eigenvalues of the pullback  action $f^*$ on the cohomology $\rm{H}^{\bullet}(X)$ can be used to detect when invariant subvarieties are translates of subtori \cite{krieger_reschke}. 
This property was proven by Raynaud \cite{raynaud} and is an occurrence of the Manin-Mumford conjecture \cite{pink_roessler}. The dynamical formulation of this conjecture was stated by Ghioca-Tucker-Zhang \cite{ghioca_tucker_zhang} and obtained in various situations \cite{zhang,baker_demarco,favre_dujardin_manin,favre2020arithmetic}.
 
\medskip

Considering the different values of the dynamical degrees and how these numbers are arranged, one can recover various ergodic properties displayed by a given rational map $f$ when it is cohomologically hyperbolic in the sense of Guedj:
\begin{equation*}
\lambda_1(f) < \lambda_2(f) < \ldots < \lambda_k(f) > \lambda_{k+1}(f) > \ldots > \lambda_g(f).
\end{equation*} 
Guedj \cite{guedj_proprietes} proved  that $f$  admits a unique measure of maximal entropy when the topological degree is the largest dynamical degree and the construction of this measure was extended  to more general cohomologically hyperbolic situations \cite{dinh_sibony_green_currents,dethelin_vigny,vigny_maxi}. 
 
 Naturally, one can ask whether one can still deduce some ergodic properties when the cohomological hyperbolicity condition is violated. 
When this happens, the eigenspace ${\rm E}_k$ associated to the spectral radius of the action on ${\rm H}^{k,k}(X)$ has dimension $e_k \geqslant 1$. 
We obtain the following  result.
\newpage

\begin{thma}\label{thm1}
 Let $X$ be an abelian variety of dimension $g\leq 4$ and let $f$ be an automorphism on $X$ of positive entropy.
Then, either $f$ is semi-conjugate to an automorphism on an abelian variety of lower dimension, or the following properties hold:
\begin{enumerate}
\item[(i)] The Haar measure is an ergodic and invariant  measure for $f$.
\item[(ii)] Take an integer $1\leqslant k \leqslant g-1$. For any $(k,k)$-form $\Omega$ (not necessarily closed), the limit in the sense of currents
\begin{equation*}
\lim_{N\rightarrow +\infty} \dfrac{1}{N} \sum_{j\leqslant N}\dfrac{(f^j)^* \Omega}{\lambda_k(f)^j}
\end{equation*} 
exists and belongs to ${\rm E}_k$.
\end{enumerate}
\end{thma}

The convergence of non-closed forms to a closed one already occurs in complex dynamics in the case of polynomial diffeomorphisms of $\C^2$ \cite{bedford_smillie} whereas the convergence of $(k,k)$ closed currents towards a unique closed positive current was known for endomorphisms having $\lambda_k(f)$ as a dominant dynamical degree \cite[Corollary 4.3.5]{dinh_sibony_super}.

The core of our proof lies in a classification of all dynamical degrees of automorphisms on simple abelian varieties up to dimension $4$ (see chapter \ref{dim4}). We obtain that an automorphism of positive entropy is either cohomologically hyperbolic or the spectral radius of the pullback action on the de Rham cohomology ${\rm H}^{1}(X)$ is a Salem number. 
The crucial step in the proof of assertion (ii) is to exploit the fact that Galois conjugates of Salem numbers that are on the unit circle have $\mathbb{Q}$-independent arguments together with the Von-Neumann ergodic theorem.


Besides knowing the shape of the dynamical degrees sequence, we ask  which algebraic numbers are dynamical degrees of automorphisms on abelian varieties. We focus on simple abelian varieties, since every algebraic integer can be realized on   non-simple ones (see Example \ref{ex_simple}) and because they give rise to non fibration preserving automorphisms \cite{oguiso_simple,amerik_campana,ghioca_scanlon}.
We shall construct  automorphisms using the algebraic structure of the endomorphism algebras. This method was used to classify  abelian surfaces having an infinite automorphisms group \cite{ghys_94} (consequently complex surfaces having infinitely many Anosov diffeomorphisms \cite{ghys_95}) and  to determine the asymptotic growth of the number of fixed-points of an endomorphism \cite{bauer_herrig,alvarado_auffarth}. 

In \cite{herrig_19}, the second author showed that a simple abelian variety $X$ whose endomorphism algebra ${\rm End}(X)\otimes \mathbb{Q}$ is a totally real number field, a totally definite quaternion algebra or a CM-field does not allow an endomorphism whose entropy is a Salem number. The remaining simple abelian varieties to consider are the ones with totally indefinite quaternion multiplication or of the second kind.
We obtain the following. 


\begin{thmb}\label{thm2} 
Let $\lambda$ be a Salem number of degree $g$. 
\begin{itemize}
\item[(i)] There exists an automorphism $f$ on a $g$-dimensional simple abelian variety $X$ with totally indefinite quaternion multiplication with dynamical degrees $$\lambda_1(f)=\ldots=\lambda_{g-1}(f)=\lambda^2\quad \textit{ and }\quad \lambda_0(f)=\lambda_g(f)=1.$$
\item[(ii)] There exists an automorphism $f$ on a $2g$-dimensional simple abelian variety $Y$ with second kind multiplication with dynamical degrees $$\lambda_1(f)=\lambda_{2g-1}(f)=\lambda^2, \quad \lambda_2(f)=\ldots=\lambda_{2g-2}(f)=\lambda^4 \quad and \quad \lambda_0(f)=\lambda_{2g}=1.$$ The endomorphism algebra ${\rm End}_\mathbb{Q}(Y)$ is a quaternion algebra whose center is a CM-field of degree $g$.
\end{itemize}
\end{thmb}


The proof of this theorem yields a more general result which is stated in Corollary \ref{thm2}.
\\
The core step in our proof is the construction of a simple abelian variety, or equivalently of a suitable divisional quaternion algebra (Theorem \ref{thm_divisional_salem}) which relies on a local-global principle for quaternion algebras combined with $\check{\textnormal{C}}$ebotarev's density theorem. 
\\
The third main result of this paper is a complete classification of the dynamical degrees of automorphisms on simple abelian varieties up to dimension $4$. Our classification draws a parallel between the endomorphism algebras ${\rm End}_{\mathbb{Q}}(X)$ and the number field $\mathbb{Q}(f)$ generated by the automorphism $f$ with the dynamical spectra $\Lambda(f)$.
 The detailed classification can be found in the tables 1 to 4 and the figures 1 to 4 in chapter \ref{dim4}.


\black

\section*{Acknowledgments}

This collaboration might have never happened without the help of Robert Lazarsfeld who put us in contact, we are thus thankful for his encouragement. We also thank Dzmitry Dudko and Alena Erchenko for helpful discussions, Tien-Cuong Dinh and Charles Favre for valuable remarks and comments and John Voight for answering questions on quaternion algebras. The second author thanks the Stony Brook Math department and Simons Center for a fruitful atmosphere and good hospitality during his stay.

\section{Abelian varieties}

We briefly introduce abelian varieties and relevant facts about their endomorphisms. More details can be found in \cite{birkenhake_lange}. 

A complex torus $X$ of dimension $g$ is a quotient group $\mathbb{C}^g/\Lambda$, where $\Lambda\subseteq\mathbb{C}^g$ is a lattice. The tori that can be embedded into the projective space $\mathbb{P}_{\mathbb{C}}^N$ are exactly \textit{abelian varieties} and they are called \textit{simple} if the only   subtori are the trivial ones $0$  and $X$.


\subsection{Endomorphisms}\label{p1}

Every holomorphic map $h\colon X\longrightarrow X$ on a $g$-dimensional complex torus $X$ is the sum of a translation and a unique endomorphism $f\colon X\longrightarrow X$ (with respect to the additive structure of $X$). 
For such an endomorphism $f$ we have the analytic and rational representation
$$\rho_a\colon {\rm End}(X)\longrightarrow {\rm M}_g(\mathbb{C}) \quad {\rm and} \quad \rho_r\colon {\rm End}(X)\longrightarrow {\rm M}_{2g}(\mathbb{Z}).$$ 
The eigenvalues of the matrices $\rho_a(f)$ and $\rho_r(f)$ will be called analytic and rational eigenvalues in the following. The two representations have the useful relation $\rho_r\otimes 1 \simeq \rho_a\oplus\overline{\rho_a}$ which can be found in a concrete expression of the Holomorphic Lefschetz Fixed-Point Formula (cf. \cite{birkenhake_lange_FP})
$$\#{\rm Fix}(f)=\det({\rm id}-\rho_r(f))=\big|\prod_{i=1}^g(1-\rho_i)(1-\overline{\rho_i})\big|,$$ where $\rho_1,\ldots,\rho_g$ are the analytic eigenvalues. We will see later how these eigenvalues determine the dynamical degrees and the entropy of the endomorphism $f$.

We will focus on endomorphisms of simple abelian varieties. By Poincar${\rm \acute{e}}$'s Complete Reducibility Theorem, the endomorphism algebra $B:={\rm End}_\mathbb{Q}(X)={\rm End}(X)\otimes \mathbb{Q}$ of a $g$-dimensional abelian variety $X$ has to be a skew field of finite $\mathbb{Q}$-dimension with a positive anti-involution $x\mapsto x'$, the Rosati involution (with respect to the polarization $L$ of $X$). Albert's classification determines the possibilities for the pairs $(B, \,')$. We denote $F$ the center of $B$ and $F_0$ the fixed field of the anti-involution $'$ restricted to $F$. A pair $(B,\,')$ is called of the first kind if $'$ is trivial on $F$ and of the second kind otherwise. 

Denote $$[B:F]=d^2, \quad [F:\mathbb{Q}]=e \quad {\rm and} \quad [F_0:\mathbb{Q}]=e_0,$$




then the classification gives the following restrictions.

\begin{prop}\label{class}{\rm (\cite[Proposition 5.5.7]{birkenhake_lange})}
\renewcommand*\arraystretch{1.4}
$$\begin{array}{|c|c|c|c|}
\hline
B={\rm End}_\mathbb{Q}(X) & d & e_0  & \text{restriction}\\
\hline
\text{totally real number field} & 1 & e & e|g \\
\hline
\text{totally indefinite quaternion algebra} & 2 & e & 2e|g \\
\hline
\text{totally definite quaternion algebra} & 2 & e & 2e|g \\
\hline
(B, \,'\,) \text{ of the second kind} & d & \frac{1}{2}e & e_0d^2|g \\
\hline

\end{array}$$
\end{prop}

The analytic eigenvalues are related to  the first fixed-point formula and to a particular norm (cf.\cite{birkenhake_lange_FP}). Let ${\rm N} \colon B\longrightarrow \mathbb{Q}$ be the reduced norm map and $f\in{\rm End}(X)$ an endomorphism of a simple abelian variety $X$, then we have $$\#{\rm Fix}(f)=\big({\rm N}(1-f)\big)^{\frac{2g}{de}}.$$

Simple abelian varieties whose endomorphism algebras are totally indefinite quaternion algebras or of the second kind are of special interest for us, we will discuss their algebraic structures in the next section.


%
%
%
%
%
%
%
%

%
%

\subsection{Quaternion algebras  and Salem numbers}


 We focus on the main properties of quaternion algebras and their commutative subalgebras.
More details on quaternion algebras can be found in \cite{voight}, \cite{gille_szamuely} and \cite{vigneras}.
\medskip

A quaternion algebra $B=\big(\frac{a,\, b}{F}\big)$ over a field $F$ (of {\rm char}$\neq 2$) with $a, b\in F^{\times}$ is an $F$-vector space with a basis $1, i, j, ij$ such that the equations $$i^2=a, \quad j^2=b \quad {\rm and} \quad ij=-ji$$ hold.

Let $F$ be a totally real number field and $\mathbb{H}=\big(\frac{-1,\,-1}{\mathbb{R}}\big)$ denote the Hamiltonian quaternions. A quaternion algebra $B$ over $F$ is called \textit{totally definite} if we have $$B\otimes_\sigma \mathbb{R}\simeq \mathbb{H}$$ for all embeddings $\sigma\colon F\hookrightarrow \mathbb{R}$. If we have $$B\otimes_\sigma\mathbb{R}\simeq {\rm M}_2(\mathbb{R})$$ for all embeddings $\sigma\colon F\hookrightarrow \mathbb{R}$, then we call $B$ \textit{totally indefinite}.

In the situation $(B,\,')$ of the second kind we will consider quaternion algebras whose center is a CM-field $K$, i.e. a quadratic extension of a totally real number field that is totally imaginary. In this case the restriction of the anti-involution $'$ to $K$ will correspond to the complex conjugation.

The endomorphism ring ${\rm End}(X)$ of a simple abelian variety $X$ is an order $\mathcal{O}$ in the endomorphism algebras $B$, i.e. a lattice with a ring structure.

We are especially interested in simple abelian varieties and their endomorphism algebras which have to be skew fields. A quaternion algebra $B=\big(\frac{a,\,b}{F}\big)$ is divisional if and only if it is not \textit{split}, i.e. the algebra $B$ is not isomorphic to the matrix ring ${\rm M}_2(F)$. To construct a divisional quaternion algebra we will later use a local-global principle which is formulated in terms of the places $v$ (we will use the same notation for the corresponding valuation) of the center $F$. We denote $F_v$ to be the completion of $F$ with respect to the valuation $v$. We say that $B$ is \textit{ramified} at $v$ if $$B_v=B\otimes_FF_v$$ is divisional. Otherwise, $B$ is called \textit{split} at $v$. In these terms, $B$ is totally definite if it ramifies at all archimedean places and totally indefinite if it splits at all archimedean places. In the following, the set ${\rm Ram}(B)$ denotes the set of ramified places of $B$. By the local-global principle for quaternion algebras (see \cite[Corollary 14.6.5]{voight}), we get the useful equivalences $$B\: {\rm divisional} \Leftrightarrow B\not\simeq {\rm M}_2(F) \Leftrightarrow {\rm Ram}(B)\neq\emptyset.$$

So, for the construction of a totally indefinite quaternion algebra $B$ we need a non-archimedean place $v$ such that $B$ ramifies at $v$.

For the mentioned fixed-point formulas we need the \textit{reduced norm} ${\rm N}\colon B\longrightarrow F$ defined by ${\rm N}(y)=y\cdot\overline{y}$, where $\overline{\alpha+\beta i+ \gamma j+ \delta ij}=\alpha-\beta i-\gamma j- \delta ij$ is the quaternion conjugation. The \textit{reduced trace} ${\rm T}(y)$ is defined by $y+\overline{y}$ such that we get the \textit{reduced characteristic polynomial} $$X^2-{\rm T}(y)X+{\rm N}(y)$$ of $y$ over $F$.
\\
\\
The endomorphism algebra of a simple abelian variety can also be a totally real number field or a CM-field. A number field $F$ is \textit{totally real} if the image of all its embeddings into the complex numbers lies in the real numbers. Further, all subfields of $F$ are again totally real (see \cite[Lemma 1.1]{herrig_19}). A \textit{CM-field} $K$ is a quadratic extension of a totally real number field $F$, but $K$ is totally imaginary, i.e., the image of no embedding of $K$ into the complex numbers lies in the real numbers. The normal closure of a CM-field is again a CM-field and subfields of a CM-field are either totally real or CM-fields (see \cite[Lemma 1.3]{herrig_19}). By Theorem 3 in \cite{daileda}, the unimodular units of CM-fields are always roots of unity.
\\
\\
The existence of Salem numbers or their algebraic conjugates in the endomorphism algebra will play an important role later. More details on these numbers can be found in \cite{bdgps}. Depending on the purpose of a paper, some authors allow real quadratic numbers to be Salem. We use the original definition introduced by Rapha\"{e}l Salem in \cite[p. 26]{salem}. A number $\lambda\in\mathbb{R}$ is a \textit{Salem number} if it is an algebraic integer larger than $1$ whose conjugates lie inside or on the unit circle, assuming that at least one conjugate actually lies on the unit circle. A number field $\mathbb{Q}(\lambda)$ is a real quadratic extension of the totally real number field $\mathbb{Q}(\lambda+\lambda^{-1})$ (see \cite[Theorem 5.2.3]{bdgps}) and obviously not a CM-field. Further, the conjugate roots of modulus $1$ are not roots of unity. 

Pisot numbers are similar to Salem number and also define a special class of numbers. A \textit{Pisot number} $\tau\in\mathbb{R}$ is an algebraic integer larger than $1$ whose conjugates lie inside the unit circle. They are not as rare as Salem numbers, because by Theorem 5.2.2 in \cite{bdgps}, every real algebraic extension of degree $n$ over the rationals contains infinitely many Pisot numbers of degree $n$, while some of them are also units.





\subsection{Dynamical degrees and entropies}

In this section, we briefly introduce the notions of topological entropy and the $k$-th dynamical degree for an endomorphism on an abelian variety. 

If $X$ is a compact K\"ahler manifold and $f\colon X\longrightarrow X$ a holomorphic map, then  Gromov \cite{gromov_03} and Yomdin \cite{yomdin_87}'s result shows that the \textit{topological entropy} ${\rm h_{top}}(f)$ (see \cite[\S 8.2]{katok}) is given by the equation $$\max_{1\leq k\leq n}\log\rho(f^*\colon\textnormal{H}^{k,k}(X)\longrightarrow \textnormal{H}^{k,k}(X))=\log(\gamma),$$ where $\rho$ stands for the spectral radius and ${\rm H}^{k,k}(X)$ for the $(k,k)$-Dolbeaut cohomology. 
\smallskip

The \textit{$k$-th dynamical degree} $\lambda_k(f)$ \textit{of $f$}  is by definition the spectral radius of $f^*$ on ${\rm H}^{k,k}(X)$.
\smallskip

In our setting, the $k$-th dynamical degree of an automorphism can be computed as follows:
 
The image of a non-zero endomorphism $f\colon X\longrightarrow X$ of an abelian variety $X$ is non-zero and defines an abelian subvariety. Considering an automorphism $f$ or a simple abelian variety $X$, the analytic representation $\rho_a(f)$ has to be an isomorphism of the universal cover and all eigenvalues are non-zero. Furthermore, the action $f^*$ of $f$ on the cohomology group ${\rm H}^1(X,\mathbb{C})$ is given by $$^t\rho_a(f)\oplus ^t\overline{\rho_a(f)}$$ which induces the action of $f$ on all cohomology groups by the isomorphism $$\bigwedge^k {\rm H}^1(X,\mathbb{C})\simeq {\rm H}^k(X,\mathbb{C}).$$
 The eigenvalues of $^t\rho_a(f)\oplus ^t\overline{\rho_a(f)}$ correspond to the eigenvalues of $\rho_r(f)$ which we denote by $\rho_1,\ldots,\rho_{2g}$ with $g=\dim X$. 
 The spectral radius $\rho$ of $f^*$ on ${\rm H}^{k,k}(X,\mathbb{C})$ is thus given by the largest product of $2k$ pairwise distinct eigenvalues $\rho_i$.
We will explain in \S \ref{p1} on how one can compute the analytic eigenvalues using the two fixed-point formulas introduced.

\medskip
%
%
 
We restrict on the properties of the dynamical degrees of automorphisms  on simple abelian varieties, the reason is that there are no obstructions when one works on non-simple abelian varieties.  

\begin{ex}\label{ex_simple}
Considering the non-simple case, we can start with an integral polynomial $$P(t)=(-1)^n(t^n+a_{n-1}t^{n-1}+\ldots+a_1t+a_0)\in\mathbb{Z}[t]$$ and its companion matrix $$f=\begin{pmatrix}0 &  & \ldots & & 0 & -a_0 \\ 1 & 0 & \ldots & & 0 & -a_1 \\  & 1 & \ddots & 0 & & \vdots \\ & & \ddots & & 0 & \vdots \\ & 0 &  & & 1 & -a_{n-1}\end{pmatrix}$$ which defines an endomorphism on the self-product $E^n$ of an elliptic curve $E$. Hence, we are able to create every algebraic integer as an analytic eigenvalue of an endomorphism on an abelian variety. Automorphisms are obviously generated by $a_0=1$.
\end{ex}

When one restricts to  simple abelian varieties with multiplication by a totally indefinite quaternion algebra and of the second kind, the situation is more constrained and we have (see \cite[Proposition 2.2, Proposition 2.3 and Proposition 2.3]{herrig_19}):  

\begin{cor}
The analytic eigenvalues of simple abelian varieties with real, totally definite quaternion or complex multiplication are either all of absolute value $1$ or none is of absolute value $1$.
\end{cor}
 
\begin{rem} To be precise with the vocabulary: By complex multiplication we mean that the endomorphism algebra is a CM-field and by simple abelian variety of the second kind we mean that the endomorphism algebra defines a division algebra over a CM-field.
\end{rem}

The  previous corollary implies the following.

\begin{cor} Let $f$ be an automorphism of a $g$-dimensional simple abelian variety with real, totally definite quaternion or complex multiplication. Then, we are in one of the following two situations:
\begin{enumerate}
\item[(i)] All dynamical degrees $\lambda_j(f)$ are $1$ and $f$ is of zero entropy.
\item[(ii)] There exists a $j\in\{1,\ldots, g\}$ such that $\lambda_j(f)$ is strictly larger than all other dynamical degrees and for all $k\in\{1,\ldots,g-1\}$ the inequality $\lambda_k(f)\neq \lambda_{k+1}(f)$ holds. In particular, $f$ is of positive entropy.
\end{enumerate}
\end{cor}

The existence of automorphisms of positive entropy is general and is a consequence of Dirichlet's unit theorem.

\begin{rem}
 Let $F$ be a totally real number field of degree $r$ or a CM-field of degree $2r$. Then, Dirichlet's unit theorem (see \cite[p. 39--44]{neukirch}) states:  Every order $\mathcal{O}$ in $F$ has exactly $r-1$ fundamental units $\epsilon_1,\ldots,\epsilon_{r-1}$, such that any unit $\epsilon$ can be written uniquely as $$\epsilon=\zeta\cdot\epsilon_1^{m_1}\cdot\ldots\cdot\epsilon_{r-1}^{m_{r-1}}$$ with a root of unity $\zeta$ and integers $m_i$. Taking such an order $\mathcal{O}$, we can construct by \cite[Chapter 9.2 and 9.6]{birkenhake_lange} a simple abelian variety $X$ with real or complex multiplication such that $\mathcal{O}$ is contained in ${\rm End}(X)$. One of the mentioned fundamental units finally corresponds to an automorphism of positive entropy.
 \end{rem}

Let us mention that some situations were realized, 
in \cite[Proposition3.6]{herrig_19}, the following automorphism on a four dimensional simple abelian variety was constructed. 

\begin{ex}
The quaternion algebra $$B=\Bigg(\frac{2,\,-2-2\sqrt{13}}{\mathbb{Q}(\sqrt{13})}\Bigg)$$ is divisional and totally indefinite. There exists a $4$-dimensional simple abelian variety $X$ whose endomorphism ring ${\rm End}(X)$ lies in $B$ and contains an automorphism $f$ whose analytic eigenvalues are the roots of the  Salem polynomial $x^4-x^3-x^2-x+1$. Denote $\lambda$ to be the Salem number of this polynomial, then we get $\lambda_1(f)=\lambda_2(f)=\lambda_3(f)=\lambda^2$ and $\lambda_4(f)=\lambda\cdot\lambda^{-1}=1$.
 \end{ex}


\section{Construction of an automorphism with prescribed entropy}

In this chapter we will prove one of the main results of this paper.

\begin{thm}\label{thm_salem_tot_indef} Let $\lambda$ be a Salem number of degree $g$. 
\begin{itemize}
\item[(i)] There exists an automorphism $f$ on a $g$-dimensional simple abelian variety $X$ with totally indefinite quaternion multiplication with dynamical degrees $$\lambda_1(f)=\ldots=\lambda_{g-1}(f)=\lambda^2\quad \textit{ and }\quad \lambda_0(f)=\lambda_g(f)=1.$$
\item[(ii)] There exists an automorphism $f$ on a $2g$-dimensional simple abelian variety $Y$ with second kind multiplication with dynamical degrees $$\lambda_1(f)=\lambda_{2g-1}(f)=\lambda^2, \quad \lambda_2(f)=\ldots=\lambda_{2g-2}(f)=\lambda^4 \quad and \quad \lambda_0(f)=\lambda_{2g}=1.$$ The endomorphism algebra ${\rm End}_\mathbb{Q}(Y)$ is a quaternion algebra whose center is a CM-field of degree $g$.
\end{itemize}
\end{thm}

We can modify the proof afterwards to get a generalization.

\begin{cor}\label{thm2} 
Let $\lambda$ be a Salem number of degree $h$. 
\begin{itemize}
\item[(i)] For all natural numbers $v$ there exists an automorphism $f$ on a $v\cdot h$-dimensional simple abelian variety $X$ with totally indefinite quaternion multiplication with dynamical degrees $$\lambda_k(f)=\lambda_{vh-k}(f)=\lambda^{2k}\quad \textit{ for } 0\leq k\leq v$$ $$\textit{ and }\quad \lambda_k(f)=\lambda^{2v} \quad \textit{ for } v<k<v(h-1).$$
\item[(ii)] For all natural numbers $v$ here exists an automorphism $f$ on a $2\cdot v\cdot h$-dimensional simple abelian variety $Y$ with second kind multiplication with dynamical degrees$$\lambda_k(f)=\lambda_{2vh-k}(f)=\lambda^{2k}\quad \textit{ for } 0\leq k\leq 2v$$ $$\textit{ and }\quad \lambda_k(f)=\lambda^{4v} \quad \textit{ for } 2v<k<2v(h-1).$$
The endomorphism algebra ${\rm End}_\mathbb{Q}(Y)$ is a quaternion algebra whose center is a CM-field of degree $v\cdot h$.
\end{itemize}
\end{cor}

The difficulty of these results comes from the fact that we have to construct a suitable endomorphism on a simple abelian variety.  Our construction is made in the following way.
\\
\\
{\bf{Construction manual:}}
\begin{enumerate}
\item We choose a Salem number $\lambda$ with minimal polynomial $P$ of degree $g$.
\item Let $\gamma$ be a complex root of the polynomial $P$. For this algebraic integer of modulus $1$ we get an imaginary quadratic extension $K=F(\gamma)$ of a totally real number field $F=\mathbb{Q}(\gamma+\overline{\gamma})$.
\item Denote by $\mathcal{O}_F$ the ring of integers of $F$. We choose an $a\in\mathcal{O}_F$ such that $K=F(\sqrt{a})$.
\item We prove the existence of a prime number $p$ such that $B=\Big(\frac{a,\, p}{F}\Big)$ becomes a divisional quaternion algebra. (The construction is explained in detail in part \ref{construct_DA}.)
\item We install a positive anti-involution $'$ on $B$ and construct an order $\mathcal{O}$ in $B$ with $\gamma\in\mathcal{O}$.
\item We prove the existence of an integer $d\in\mathbb{Z}_{>0}$ such that the quaternion algebra $B_2=\Big(\frac{a, \, p}{F(\sqrt{-d})}\Big)$ stays divisional. We show that the positive anit-involution $'$ can be extended to $B_2$ and acts as complex conjugation on $F(\sqrt{-d})$. We also show the existence of an order $\mathcal{O}_2$ in $B_2$ that contains $\gamma$.
\item The following two propositions provide a simple $g$-dimensional abelian variety $X$ and a simple $2g$-dimensional abelian variety $Y$ such that the endomorphism rings ${\rm End}(X)$ and ${\rm End}(Y)$ contain $\mathcal{O}$ resp. $\mathcal{O}_2$.

\begin{prop}{\rm(\cite[Chapter 9.4]{birkenhake_lange})}\label{prop_construction_abelian_variety_quat_algebra} Let $B$ be a totally indefinite quaternion algebra over a totally real number field $F$ with $[F:\mathbb{Q}]=e$ and $'$ a positive anti-involution on $B$. Fix an order $\mathcal{O}$ in $B$ and suppose that $B$ is divisional. 
Then there exists a $2e$ dimensional simple abelian variety $X$ whose endomorphism ring ${\rm End}(X)$ contains $\mathcal{O}$.
\end{prop}

\begin{prop}{\rm (\cite[Chapter 9.6]{birkenhake_lange})}\label{construction_AV_CM_Quat} Let $B'$ be a divisional quaternion algebra over a CM-field $F(\sqrt{-d})$ with $[F(\sqrt{-d}):\mathbb{Q}]=2e$ and $'$ a positive anti-involution on $B_2$. Fix an order $\mathcal{O}_2$ in $B_2$ and suppose that $B_2$ is divisional. Then there exists a $2e$ dimensional simple abelian variety $Y$ whose endomorphism ring ${\rm End}(Y)$ contains $\mathcal{O}_2$.
\end{prop}
\item We finally get automorphisms $f$ and $f_2$ in $\mathcal{O}$ resp. $\mathcal{O}_2$ whose rational eigenvalues are $2$-times resp. $4$-times the roots of the Salem polynomial $P$ which leads to the claimed dynamical degrees.
\end{enumerate}



In the next section, we prove the statement of part 4 of our construction. After this crucial part we are able to prove Theorem \ref{thm_salem_tot_indef} in a next section.

\subsection{Construction of a divisional quaternion algebra}\label{construct_DA}

The main result of this section is the construction of a certain divisional quaternion algebra.


\begin{thm}\label{thm_divisional_salem} Let $F$ be a totally real number field and $K=F(\sqrt{a})$ a quadratic extension for $a\in\mathcal{O}_F$. Then there exists a prime number $p$ such that the quaternion algebra $B=\big(\frac{a,\, p}{F}\big)$ is divisional.
\end{thm}


To prove this theorem we will use the local-global principle for quaternion algebras which tells us that a quaternion algebra $B$ over a global field $F$ is divisional if and only if the set of ramified places ${\rm Ram}(B)$ of $B$ is not empty. 

Starting with the totally real number field $F$ and the algebraic integer $a\in\mathcal{O}_F$, we will look for a prime number $p$ and a non-archimedean valuation $v_F$ extending the $p$-adic valuation such that $B=\big(\frac{a,\, p}{F}\big)$ ramifies at $v_F$. We will construct the valuation $v_F$ to meet the assumptions of the following proposition.


\begin{prop}\label{quat_div}{\rm(\cite[Corollary 12.3.9]{voight})} Let $F_{v_F}$ be a non-archimedean local field, $p$ a uniformizer for $v_F$ and $k_{v_F}$ the residue field with ${\rm char}\neq 2$. Then, a quaternion algebra $B_0$ is divisional if and only if $B_0\simeq \big(\frac{a,\, p}{F_{v_F}}\big)$, where $v_F(a)=0$ and $a$ is nontrivial in $k^\times_{v_F}/(k^{\times}_{v_F})^2$.
\end{prop}

The following two statements provide us the right amount of appropriate prime ideals $\mathfrak{p}$ in $F$ and $\mathfrak{P}$ in $K$ to construct the required valuation $v_F$.

Let $F\subseteq K$ be a Galois extension with Galois group $G:={\rm Gal}(K/F)$. For every $\sigma\in G$, we consider the set $M:={\rm P}_{K/F}(\sigma)$ of prime ideals $\mathfrak{p}$ of $F$, unramified in $K$ such that there exists a prime ideal $\mathfrak{P}$ of $K$ above $\mathfrak{p}$ whose Frobenius automorphism $\Big(\frac{K/F}{\mathfrak{P}}\Big)$ over $F$ coincides with $\sigma$. The \textit{Dirichlet density} of $M$, provided it exists, is defined as $${\rm d}(M)=\lim_{s\to 1^+} \frac{\sum_{\mathfrak{p}\in M}\mathfrak{N}(\mathfrak{p})^{-s}}{\sum_\mathfrak{p}\mathfrak{N}(\mathfrak{p})^{-s}},$$ where $\mathfrak{N}(\mathfrak{p})$ stands for the index $[\mathcal{O}_F:\mathfrak{p}]$.

\begin{thm}{\rm($\check{\textnormal{C}}$ebotarev's density theorem \cite[Theorem 13.4]{neukirch})}\label{cebotarev} Let $K/F$ be a Galois extension with group $G$. Then for every $\sigma\in G$, the set ${\rm P}_{K/F}(\sigma)$ has  a density, denoted ${\rm d}({\rm P}_{K/F}(\sigma))$, and it is given by the formula $${\rm d}({\rm P}_{K/F}(\sigma))=\frac{\#\langle\sigma\rangle}{\# G}.$$ 
\end{thm}

The following proposition provides us infinitely many prime ideals that are unramified in a separable field extension:

\begin{prop}\label{infinte_prime}{\rm(\cite[Poposition 8.4]{neukirch})} If $ K\vert F$ is a separable field extension, then there are only finitely many ideals of $F$ that are ramified in $K$.
\end{prop}

In our setting, $K$ is a degree $2$ extension of $F$. We now choose suitable prime ideals in the fields $F$ and $K$ above a prime number $p$ to construct a non-archimedean valuation $v_K$ on $K$. In the prove of Theorem \ref{thm_divisional_salem}, this $v_K$ will restrict to the desired valuation $v_F$ and deliver all required properties to apply Proposition \ref{quat_div}.
\begin{cor} \label{cor_chebotarev} Let $F$ be a totally real number field and $K= F(\sqrt{a})$ an imaginary quadratic extension for $a\in\mathcal{O}_F$. 
Then there exists a prime number $p$, a valuation $v_K$ on $K$, a prime ideal $\mathfrak{p}$ in $F$ containing $p\Z$ and a prime ideal $\mathfrak{P}$ in $K$ containing $\mathfrak{p}$ satisfying the following conditions:
\begin{enumerate}
\item[(i)] $p$ and ${\rm N}_{F/\mathbb{Q}}(a)$ are coprime. 
\item[(ii)] The ideal $\mathfrak{p} \subset F$ is unramified in $K$. 
\item[(iii)] The Frobenius automorphism $\Big(\frac{K/F}{\mathfrak{P}}\Big)$ of $\mathfrak{P}$ corresponds to the complex conjugation.
\item[(iv)] For all $\alpha \in \mathfrak{P} \setminus \mathfrak{P}^2,$ $v_K(\alpha) = 1$ and  the restriction of $v_K$ to $\Z$ is the $p$-adic valuation.
\end{enumerate}
\end{cor}

Observe that the last condition implies that $v_K$ extends the $p$-adic valuation on $\Z$ since the ideal $\mathfrak{P}$ contains the ideal $p\Z$.

\begin{proof}
We start with $\sigma\in{\rm Gal}(K/F)$ that coincides with the complex conjugation on $K$. By Theorem \ref{cebotarev} and Proposition \ref{infinte_prime}, we get infinitely many pairs $(\mathfrak{p},\mathfrak{P})$ of prime ideals, such that $\mathfrak{p}$ is unramified in $K$, $\mathfrak{P}$ in $K$ lies above $\mathfrak{p}$ and the Frobenius automorphism $\Big(\frac{K/F}{\mathfrak{P}}\Big)$ of $\mathfrak{P}$ corresponds to $\sigma$. For all these pairs $(\mathfrak{p},\mathfrak{P})$ of prime ideals, the conditions (ii) and (iii) hold.

Next, we look at the factorization $\prod_{j=1}^kp_j^{d_j}$ of the norm ${\rm N}_{F/\mathbb{Q}}(a)$ in $\mathbb{Z}$, as $a$ lies in $\mathcal{O}_F$. For every prime $p_j$ we denote the unique decomposition $p_j\mathcal{O}_F=\prod_{l=1}^m\mathfrak{p}_l^{e_l}$ of prime ideals in $\mathcal{O}_F$. The number of all  prime ideals that occur in these decompositions is finite.

At least, we choose one of our infinitely many pairs $(\mathfrak{p},\mathfrak{P})$ such that $\mathfrak{p}$ is not a factor in the decomposition $p_j\mathcal{O}_F$ as $\prod_{l=1}^m\mathfrak{p}_l^{e_l}$ for all $j\in\{1,\ldots,k\}$. 
By construction, such a prime ideal $\mathfrak{p}$ lies above  a prime number $p$ which is coprime to ${\rm N}_{F/\mathbb{Q}}(a)$. 
This shows that we can find a prime number $p$ and a pair $(\mathfrak{p}, \mathfrak{P})$ of prime ideals satisfying the conditions $(i), (ii)$ and $ (iii)$. 
By Proposition \ref{infinte_prime}, there are only finitely many prime numbers $p$ such that the ideal $p\Z$ is ramified in $F$. In particular, we can choose $(p , \mathfrak{p}, \mathfrak{P})$ so that $p\Z$ is also unramified in $F$. 
\smallskip

Let us now construct a valuation $v_K$ satisfying condition $(iv)$. 
We thus take $v_K$ to be the unique discrete valuation such that 
$$v_K(\alpha) = 1 ,$$
for all $\alpha \in \mathfrak{P} \setminus \mathfrak{P}^2$.
We now check that the above valuation $v_K$ satisfies condition $(iv)$.
Since the ideal $\mathfrak{p} \subset F$ is unramified in $K$, it can be decomposed as:
\begin{equation*}
\mathfrak{p} \mathcal{O}_K = \prod_{j=0}^l I_j,
\end{equation*}
where $I_j$ are prime ideals of $K$. 
Since the ideal $\mathfrak{P}$ contains $\mathfrak{p}$, we can suppose that $I_0 = \mathfrak{P}$. 

Similarly, since we have constructed $p$ so that the ideal $p\Z$ is unramified in $F$, we have:
\begin{equation*}
(p )\mathcal{O}_F = \mathfrak{p} \prod_{j=0}^r J_j,
\end{equation*}
where $J_j$ are prime ideals in $F$.
We obtain a decomposition $(p)$ in $K$ as:
\begin{equation*}
(p) \mathcal{O}_K = \mathfrak{P} \prod_{j=1}^l I_j \prod_s \tilde{J}_{s},
\end{equation*}
where $\tilde{J}_s$ are prime ideals of $K$ containing one of the ideals $J_j$.
This shows that $v_K(p)=1$ and condition $(iv)$ holds.
\end{proof}

\begin{proof}[Proof of Theorem \ref{thm_divisional_salem}]
Consider the quadruple $(p, v_K , \mathfrak{p}, \mathfrak{P})$ constructed using Corollary \ref{cor_chebotarev}.
Denote by $k_F:=(\mathcal{O}_F)_\mathfrak{p}/(\mathfrak{p}(\mathcal{O}_F)_\mathfrak{p})\simeq\mathcal{O}_F/\mathfrak{p}$ the residue field with respect to $v_K$ restricted to $F$. The following lemma will translate the results of Corollary \ref{cor_chebotarev} to meet the assumptions of Proposition \ref{quat_div}.

\begin{lem} \label{lem_divisional_technical} The following properties are satisfied.
\begin{enumerate}
\item[(i)] One has $v_K(a) = 0$ and $a$ is nontrivial in the residue field $k_F$ of $F$. 
\item[(ii)]The element $p$ is a uniformizer for $v_F$. 
\item[(iii)] One has  $a \in k_F^* \setminus (k_F^*)^2$. 

\end{enumerate}
\end{lem}

\begin{proof}[Proof of Lemma \ref{lem_divisional_technical}]
By Corollary \ref{cor_chebotarev} (i), $p$ and ${\rm N}_{F/\mathbb{Q}}(a)$ are coprime, so $v_K(a)=0$ and assertion $(i)$ holds.
 
As for assertion $(ii)$, it follows directly from assertion $(ii)$ and $(iv)$ of Corollary \ref{cor_chebotarev}.
\smallskip

Let us prove $(iii)$. We consider the Galois extension $K/F$ with Galois group $G$  and the decomposition $\mathfrak{p}\mathcal{O}_K=(\prod_{j=1}^g\mathfrak{P}_j)^{e}$ of $\mathfrak{p}$ in $K$, where $\mathfrak{P_j}=\mathfrak{P}$ holds for one $j$. Besides $k_F$ we get the residue field $k_K:=\mathcal{O}_K/\mathfrak{P}$ and denote $[k_K:k_F]=f$. Altogether, we get the equation $[K:F]=e\cdot f \cdot g=2$. 
Our aim is to show that $f=2$ in our setting. 

The \textit{decomposition group} $D_{\mathfrak{P}}:=\{\sigma\in G \, |\, \sigma(\mathfrak{P})=\mathfrak{P}\}$ of $\mathfrak{P}$ is of order $e\cdot f$ and the \textit{inertia group} $I_{\mathfrak{P}}:={\rm ker}(D_{\mathfrak{P}}\rightarrow {\rm Gal}(k_K/k_F))$ is of order $e$. 
The Frobenius automorphism $\Big(\frac{K/F}{\mathfrak{P}}\Big)$ of $\mathfrak{P}$ can be identified with an element of the Galois group ${\rm Gal}(k_K/ k_F)$  via the isomorphism ${\rm Gal}(k_K/k_F)\simeq D_{\mathfrak{P}}/I_{\mathfrak{P}}$ (see \cite[Chapter 9, p. 57]{neukirch}). In our case, this identifies the complex conjugation $\sigma \in D_{\mathfrak{P}} $ with the corresponding Frobenius automorphism. 
\smallskip 

\textbf{Claim}: The inertia group $I_\mathfrak{P}$ is trivial. 
\smallskip

Since $p$ is a prime number distinct from $2$, the residue field $k_K$ is separable over $k_F$. As a result, the fact that $\mathfrak{p}$ is unramified in $K$ implies that the inertia group $I_\mathfrak{P}$ is trivial by (see \cite[Chapter 9, p. 58]{neukirch}), and the claim is proved.
\medskip

Using the claim, we have $\sigma \in D_\mathfrak{P} \simeq D_{\mathfrak{P}}/ I_\mathfrak{P} \simeq {\rm Gal}(k_K/k_F)$ and this shows that the order of the Frobenius automorphism corresponding to $\sigma$ is of order exactly $2$ in ${\rm Gal}(k_K/k_F)$. Hence the order of $D_{\mathfrak{P}}/I_{\mathfrak{P}}=f$ must be a multiple of $2$, which implies $f =2$ and $e=g=1$, since $e\cdot f\cdot g =2$ holds. 
Moreover, the isomorphism between $D_{\mathfrak{P}}/I_{\mathfrak{P}}$ and the Galois group of $k_K/k_F$ implies that   the degree of the extension $[k_K:k_F]$ is also $2$, so the residue class of $\sqrt{a} $ in $k_K$ does not belong to $k_F$ and determines a quadratic extension of this field.
We have thus shown that property $(iii)$ holds.


\end{proof}
We denote $F_{v_F}$ to be the completion of $F$ for $v_F$ with residue field $k_{v_F}$. Since $v_F$ is a normalized and discrete valuation on $F$, the extension of $v_F$ to $F_{v_F}$ is again normalized and discrete. Hence, $p$ stays a uniformizer of the extended valuation. Because of the isomorphism $k_F \simeq k_{v_F}$, a nontrivial and square-free element in the residue field $k_F$ keeps these properties in $k_{v_F}$.

Finally, Lemma \ref{lem_divisional_technical} and Proposition \ref{quat_div} imply Theorem \ref{thm_divisional_salem}.
\end{proof}


\subsection{Proof of Theorem \ref{thm_salem_tot_indef} and Corollary \ref{thm2}}

\begin{flushleft}{\bf{Proof of part (i) of Theorem \ref{thm_salem_tot_indef}:}}\end{flushleft}

Let $\gamma$ be a complex root of a Salem polynomial of degree $g$ and $\lambda$ the corresponding Salem number. Then $K=\mathbb{Q}(\gamma)$ defines an imaginary quadratic extension of a totally real number field $F=\mathbb{Q}(\gamma+\overline{\gamma})$ with $[F:\mathbb{Q}]=g/2$. To apply Theorem \ref{thm_divisional_salem} we first look for an appropriate generator of the field $K$ over $F$.

\begin{lem} There exists an algebraic integer $a\in\mathcal{O}_F$ such that $F(\sqrt{a})=K$.
\end{lem} 
\begin{proof} We can write $\gamma$ as $\alpha+\beta \sqrt{b}$ with $\alpha,\beta \in F$ and $b\in F_{<0}$ and it is known that $\gamma+\overline{\gamma}=2\alpha$ lies in $\mathcal{O}_F$. For the norm of $\gamma$ we get the equation $${\rm N}_{K/F}(\gamma)=\alpha^2-\beta^2b\in\mathcal{O}_F.$$
Because of $$-(4\alpha^2-4\beta^2-(2\alpha)^2)=4\beta^2b\in\mathcal{O}_F$$ we obtain an algebraic integer $a=4\beta^2b$ such that $\sqrt{a}=2\beta\sqrt{b}$ generates $K$ over $F$.
\end{proof}

Now, we have the totally real number field $F$ and the quadratic extension $K=F(\sqrt{a})$ with $a\in\mathcal{O}_F$. Given these data, we can apply Theorem \ref{thm_divisional_salem} and get a prime number $p$ such that the quaternion algebra $B=\Big(\frac{a\, ,\, p}{F}\Big)$ is divisional. Since $\sigma(p)=p>0$ for every embedding $\sigma\colon F\hookrightarrow \mathbb{R}$, we always get $B\otimes_\sigma\mathbb{R}\simeq{\rm M}_2(\mathbb{R})$ which means by definition that $B$ is totally indefinite.

To show that $\gamma$ defines an automorphism of a simple abelian variety, we have to show that $\gamma$ lies in an order $\mathcal{O}$ of the quaternion algebra $B$, while $B$ is the endomorphism algebra of a simple abelian variety.

\begin{lem}\label{involution} There exists an order $\mathcal{O}$ in $B=\big(\frac{a,\, p}{F}\big)$, such that we have $\gamma\in\mathcal{O}$ and $\mathcal{O}$ is contained in the endomorphism ring ${\rm End}(X)$ of a $g$-dimensional simple abelian variety $X$. 
\end{lem}

\begin{proof}
Let $\mathcal{O}_K$ be the ring of integers in the field $K$ and we write $j^2=p$ and $i^2=a$ for the defining elements of $B$. Via the embedding $\iota\colon K\hookrightarrow B$ we get a $\mathbb{Z}$-lattice $\mathcal{O}:=\mathcal{O}_K\oplus \mathcal{O}_Kj$ in $B$ and by easy calculations we see that $\mathcal{O}$ is also a $\mathbb{Z}$-order. Since $\gamma$ is an algebraic integer in $K$, we have $\gamma\in \mathcal{O}_K\subseteq\mathcal{O}$.
\\
On the quaternion algebra $B$ we have a positive anti-involution given by $x\mapsto i^{-1}\overline{x}i$, where $\overline{x}$ is the quaternion conjugation. We finish the proof of the lemma by using Proposition \ref{prop_construction_abelian_variety_quat_algebra}.
\end{proof}

We now determine the eigenvalues of the endomorphism $\gamma$ which are needed to compute the dynamical degrees. Since ${\rm N}_{B/\mathbb{Q}}(\gamma)=1$, the map $\gamma$ is in fact an automorphism. We start by considering the reduced characteristic polynomial $$\chi_{\rm red}(\gamma)(x)=x^2-{\rm T}_{B/F}(\gamma)\cdot x+ {\rm N}_{B/F}(\gamma)$$ of $\gamma$ over the center $F$. With regard to the embedding $\iota\colon K\hookrightarrow B$, we consider $\gamma$ as an element of $B$.
We have the relation $${\rm N}_{B/F}(x-\gamma)=\chi_{\rm red}(\gamma)(x).$$
Our element $\gamma=\alpha+\beta_0\sqrt{a}$ leads under $\iota$ to $\gamma=\alpha+\beta_0 i$ and we have ${\rm N}_{B/F}(\gamma)=\alpha^2-\beta_0^2a=1$. Together with the equation ${\rm T}_{B/F}(\gamma)=2\alpha$ we receive the factorization 
$$\chi_{\rm red}(\gamma)(x)=(x-\alpha+\beta_0\sqrt{a})(x-\alpha-\beta_0\sqrt{a})=(x-t_1)(x-t_2)$$
with $t_1,t_2\in K\setminus F$.

We have $${\rm N}_{B/\mathbb{Q}}(x-\gamma)={\rm N}_{F/\mathbb{Q}}(x-t_1)(x-t_2)$$ and further 
\begin{equation}\label{norm}{\rm N}_{B/\mathbb{Q}}(x-\gamma)^2={\rm N}_{F/\mathbb{Q}}((x-t_1)(x-t_2))^2={\rm N}_{K/\mathbb{Q}}(x-t_1)(x-t_2).\end{equation}

For integers $s$ we look at the endomorphisms $\gamma-s+1$ and compare the two fixed-point formulas. The second one tells us $$\#{\rm Fix}(\gamma-s+1)={\rm N}_{B/\mathbb{Q}}(s-\gamma)^{2g/g}.$$
Let $\sigma_1,\ldots,\sigma_g$ be the $\mathbb{Q}$-embeddings of $K$ into the algebraic closure. With the previous ideas we get the equation
\begin{equation}\label{fp-formulas}\prod_{k=1}^g(s-\sigma_k(t_1))\prod_{k=1}^g(s-\sigma_k(t_2))=\#{\rm Fix}(\gamma-s+1)=\prod_{l=1}^g(s-\rho_l)(s-\overline{\rho_l}).\end{equation}
This equation holds for all integers $s$ and $t_1$ and $t_2$ have the same minimal polynomial over $\mathbb{Q}$, since they are complex conjugate to each other and of modulus $1$. Hence, all analytic eigenvalues $\rho_l$ have to coincide with a root of the minimal polynomial of $\gamma$. Finally, the dynamical degree $\lambda_k(\gamma)$ is the product of $2k$ pairwise distinct (with regard to the numeration) elements $\rho_1,\ldots,\rho_g,\overline{\rho_1},\ldots,\overline{\rho_g}.$ which completes the proof. 
\\
\begin{flushleft}{\bf{Proof of part (ii) of Theorem \ref{thm_salem_tot_indef}:}}\end{flushleft}

We consider the divisional quaternion algebra $B=\big(\frac{a\, ,p}{F}\big)$ from the previous part. Our strategy is to find a suitable integer $d\in\mathbb{Z}_{>0}$ such that the quaternion algebra $B_2=\big(\frac{a\, ,p}{F(\sqrt{-d})}\big)$ stays divisional. This is equivalent to the following lemma. 



\begin{lem} There exists an integer $d\in\mathbb{Z}_{>0}$ such that there exists no embedding $F(\sqrt{-d})\hookrightarrow B$ as an $F$-algebra.
\end{lem}
\begin{proof} 

If $v$ is a non-archimedean norm on $F$, we denote by $F_v$ its completion with respect to this norm.

Suppose by contradiction that for all integer $d\neq 0$, there exists an embedding $F(\sqrt{-d}) $ into $B$. 
Let us take the valuation $v_K|_F=v_F\in {\rm Ram}(B)$ and the prime $p$ given by Corollary \ref{cor_chebotarev}. By construction, $p$ is a uniformizer for $v_F$ which is an extension of the $p$-adic valuation. That $F(\sqrt{-d})$ embeds into $B$ for all integer $d \neq 0$ is equivalent to $-d \not\in F_v^{\times^2}$ for all these $d$ and all $v \in {\rm Ram}(B)$.
Furthermore, there is a natural inclusion from $\mathbb{Z}_p^\times$ to $F_{v_F}^\times$ and we have a canonical isomorphism $\mathbb{Q}_p^\times \simeq \mathbb{Z}\times \mathbb{Z}_p^\times$. We look at the image of the canonical projection of $\mathbb{Z}_p$ onto its residue field $\mathbb{Z}/p \mathbb{Z}$.
Because the canonical projection onto the residue field induces the  isomorphism between  $\mathbb{Z}_p^\times /(\mathbb{Z}_p^\times)^2  $ and $(\mathbb{Z}/p\mathbb{Z})^\times / (\mathbb{Z}/p\mathbb{Z}^\times)^2 $, we deduce that an integer $-d$ which is coprime to $p$ is a square if and only if its image in the residue field is also a square. 
As a result, there exist one integer $d$ such that the image of $-d$ in the residue field is a square, which is a contradiction. 
\end{proof}

The positive anti-involution $'$ from Lemma \ref{involution} extends to $B_2$ by acting on $F(\sqrt{-d})$ as complex conjugation. As in Lemma \ref{involution} we get an order $\mathcal{O}_2=\mathcal{O}_{F(\sqrt{-d})}\oplus\mathcal{O}_{F(\sqrt{-d})}j$ in $B_2$, where $\mathcal{O}_{F(\sqrt{-d})}$ is the ring of integers in $F(\sqrt{-d})$ and $j^2=p$. Now, we can apply Proposition \ref{construction_AV_CM_Quat} and get:

\begin{lem}
There exists an order $\mathcal{O}_2$ in $B_2=\big(\frac{a \, ,p}{F(\sqrt{-d})}\big)$, such that we have $\gamma\in\mathcal{O}_2$ and $\mathcal{O}_2$ is contained in the endomorphism ring ${\rm End}(X)$ of a $2g$-dimensional simple abelian variety.
\end{lem}

We can continue as in the proof of part (i) such that we get $${\rm N}_{B_2/\mathbb{Q}}(x-\gamma)^2={\rm N}_{F(\sqrt{-d})/\mathbb{Q}}(x-t_1)(x-t_2)^2$$ $$={\rm N}_{K(\sqrt{-d})/\mathbb{Q}}(x-t_1)(x-t_2)={\rm N}_{K/\mathbb{Q}}(x-t_1)(x-t_2)^2$$ as equation (\ref{norm}).

Having the same notation as in part (i), equation (\ref{fp-formulas}) becomes 
$$\big(\prod_{k=1}^g(s-\sigma_k(t_1))\prod_{k=1}^g(s-\sigma_k(t_2))\big)^2=\prod_{l=1}^{2g}(s-\rho_l)(s-\overline{\rho_l})$$ and the statement about the dynamical degrees follows in the same way as above.

\begin{flushleft}{\bf{Proof of Corollary \ref{thm2}:}}\end{flushleft}

Let $\gamma$ be a complex root of a Salem polynomial of degree $h$ and $\lambda$ the corresponding Salem number. The number field $\mathbb{Q}(\gamma)$ is an imaginary quadratic extension of the totally real number field $\mathbb{Q}(\gamma+\overline{\gamma})$ which is of degree $h/2$. Taking a natural number $v$, we extend this totally real number field to a totally real number field $F$ of degree $v\cdot h/2$ and also get a quadratic extension $K=F(\gamma)$. Changing the notation from $c\cdot h/2$ to $g/2$, we can proceed as in the proofs of part (i) and (ii) of Theorem \ref{thm_salem_tot_indef}. The only differences will be new exponents in the equations \ref{fp-formulas} of the mentioned proofs which become 
$$
\big(\prod_{k=1}^h(s-\sigma_k(t_1))\prod_{k=1}^h(s-\sigma_k(t_2))\big)^{v}=\#{\rm Fix}(\gamma-s+1)=\prod_{l=1}^{vh}(s-\rho_l)(s-\overline{\rho_l})
$$ 
$$ \textnormal{and}$$
$$
\big(\prod_{k=1}^h(s-\sigma_k(t_1))\prod_{k=1}^h(s-\sigma_k(t_2))\big)^{2v}=\#{\rm Fix}(\gamma-s+1)=\prod_{l=1}^{2vh}(s-\rho_l)(s-\overline{\rho_l}).
$$
Considering the analytic eigenvalues, we get the statements about the dynamical degrees.

\begin{rem} The equations \ref{fp-formulas} of the previous proofs show that the degree of the minimal polynomial of an analytic eigenvalue has to divide the dimension of the abelian variety. The consequence for occurring Salem numbers is:
\\
 Let $X$ be a $g$-dimensional simple abelian variety whose endomorphism algebra ${\rm End}_\mathbb{Q}(X)$ is either a totally indefinite quaternion algebra or a quaternion algebra over a CM-field. If $X$ admits an automorphism $f$ whose analytic eigenvalues contain a Salem number $\lambda$, then the degree $h$ of $\lambda$ divides the dimension $g$.
\end{rem}

\subsection{Application: Minimal entropy and Salem numbers in dimension four}

In this part, we apply our construction to get an automorphism on a simple abelian fourfold whose entropy is a Salem number. This will be the smallest Salem number that can occur as the entropy in dimension $4$.

The minimal Salem number of degree $4$ is $\lambda:=1/4(1+\sqrt{13}+\sqrt{2\sqrt{13}-2})$ and its minimal polynomial is of the form

$$S(X)=X^4-X^3-X^2-X+1 = (X - \lambda) (X - \lambda^{-1}) (X - \gamma) (X- \bar \gamma),$$ 
where $\gamma \in \C$ is a point on the unit circle given by
 $$\gamma := \dfrac{1 - \sqrt{13} + i \sqrt{2 + 2\sqrt{13}}}{4}.$$

Let us construct an automorphism $f$ on a simple $4$-dimensional abelian variety with totally indefinite quaternion multiplication whose first, second and third dynamical degrees are $\lambda^2$. 
Following our construction above, the final map $f$ will coincide with the multiplication by $\gamma$ and its analytic eigenvalues will be the roots of the Salem polynomial $S$.

Set $\alpha = 1 + \sqrt{13}$ and  $a = - 2 \alpha$, our totally real number field $F$ is $\mathbb{Q}(\gamma+\overline{\gamma})=\mathbb{Q}(\sqrt{13})$ and its imaginary quadratic extension $K$ is $F(\sqrt{- 2 \alpha}) = F(\sqrt{ a})$. 
We determine a suitable prime number $p$, such that $B:=\Big(\frac{a,\, p}{F}\Big)$ is a skew field; $B$ is a totally indefinite quaternion algebra, since $p$ is a positive integer.
To do so, we shall apply the following proposition.

\begin{prop} \label{prop_salem_4} The following properties hold.
\begin{enumerate}
\item[(i)] The norm of $a$ over $\mathbb{Q}$  is ${\rm N}_{K/\mathbb{Q}}(a)=-2^4\cdot 3$.
\item[(ii)] The ring of integers of $F$ is $\mathcal{O}_F=\mathbb{Z}\oplus\frac{1+\sqrt{13}}{2}\mathbb{Z}$.
\item[(iii)] The norm of $\sqrt{a}/2$ is ${\rm N}_{K/\mathbb{Q}}(1+\frac{i\sqrt{2+2\sqrt{13}}}{2})={\rm N}_{F/\mathbb{Q}}(\frac{3+\sqrt{13}}{2})=-1$.
\end{enumerate}
\end{prop} 
 
 \begin{lem} The ring of integers $\mathcal{O}_K$ of $K$ is given by $$\mathcal{O}_K=\mathbb{Z}\oplus\mathbb{Z}\gamma\oplus\mathbb{Z}\gamma^2\oplus{Z}\gamma^3$$ and the discriminant $\Delta_K$ of $\mathcal{O}_K$ is $-523$.
 \end{lem}
 \begin{proof} The element $\gamma$ is an algebraic integer and its minimal polynomial is $$S(X):=X^4-X^3-X^2-X+1.$$  Further, $1,\gamma, \gamma^2$ and $\gamma^3$ are algebraic integers, too. To prove the statement it is sufficient to show that the discriminant of these four elements is squarefree (see \cite[Corollary 3.33]{jarvis}). The discriminant $\Delta\{1, \gamma, \gamma^2, \gamma^3\}$ can be computed ( see \cite[Proposition 3.31]{jarvis}) by $${\rm N}_{K/\mathbb{Q}}(S'(\gamma)).$$ Using computer support, we get $$S'(\gamma)=\frac{1}{4}\Big(3-\sqrt{13}-i\sqrt{6(\sqrt{13}-1)}\Big).$$ We calculate the norm $${\rm N}_{K/\mathbb{Q}}\Big(\frac{1}{4}\Big(3-\sqrt{13}-i\sqrt{6(\sqrt{13}-1)}\Big)\Big)={\rm N}_{F/\mathbb{Q}}\Big(\frac{1}{2}\Big(51+19\sqrt{13}\Big)\Big)=-523$$ which is squarefree and hence proves the claim.
 \end{proof}
 Since the ideal $(-523)\mathcal{O}_K$ is not contained in the ideal $(5)\mathcal{O}_K$, we get the following result (see \cite[Corollary 2.12]{neukirch}).
 
 \begin{cor} The prime ideal $(5)\mathbb{Z}$ of $\mathbb{Z}$ is unramified in $K$.
 \end{cor}
 
 The prime ideal $(5)\mathbb{Z}$ is unramified in $K$ which means $$(5)\mathcal{O}_K=\mathfrak{p}^{e_1}_1\cdot\ldots\cdot\mathfrak{p}^{e_k}_k$$ with $e_1=\ldots=e_k=1$. Since $\mathcal{O}_F\subseteq\mathcal{O}_K$ holds, the equation $$(5)\mathcal{O}_F\mathcal{O}_K=(5)\mathcal{O}_K$$ implies the following.
 
 \begin{cor} The prime ideal $(5)\mathcal{O}_F$ of $\mathcal{O}_F$ is unramified in $K$.
 \end{cor}
%
\begin{lem} \label{lem_ideal_5_technical} The ideal $5 \mathcal{O}_K$ is a prime ideal of the ring $\mathcal{O}_K$. 
\end{lem} 
 
\begin{proof}
 We show that for any two integers $u,v \in \mathcal{O}_K$,  the product $u  v $ belongs to $5\mathcal{O}_K$ if $u$ or $v$ is in the ideal $5 \mathcal{O}_K$. 
Let us write $u ,v $ in the basis given by Proposition \ref{prop_salem_4}.  
\begin{equation*}
u = u_1 + u_2 \gamma  +  u_3 \gamma^2 + u_4 \gamma^3,
\end{equation*}

\begin{equation*}
v = v_1 + \dfrac{\alpha v_2}{2} +  \dfrac{i \sqrt{2 \alpha} v_3}{2} + \dfrac{i v_4  \alpha \sqrt{2 \alpha}  }{4},
\end{equation*}
where $u_i, v_i$ are integers.
The product $u  v$ is then of the form:
\begin{multline*}
u v = u_1 v_1 - u_4 v_2 - u_3 v_3 - u_4 v_3 - u_2 v_4 - u_3 v_4 - 2 u_4 v_4 \\
 +  \gamma (u_2 v_1 + u_1 v_2 + u_4 v_2 + u_3 v_3 + u_2 v_4 + u_4 v_4) \\
 + \gamma^2 (u_3 v_1 + u_2 v_2 + u_4 v_2 + u_1 v_3 + u_3 v_3 + 2 u_4 v_3 + u_2 v_4 + 
    2 u_3 v_4 + 2 u_4 v_4)  \\
 + \gamma^3 (u_4 v_1 + u_3 v_2 + u_4 v_2 + u_2 v_3 + u_3 v_3 + 2 u_4 v_3 + u_1 v_4 + 
    u_2 v_4 + 2 u_3 v_4 + 4 u_4 v_4).
\end{multline*}
One sees that the conditions that $u v \in 5 \mathcal{O}_K$ holds implies that $u_i, v_i$ are solutions of the following system of equations in $\Z/5\Z$:
\begin{equation*}
\left  \{ \begin{array}{llll}
L_1(u,v):= u_1 v_1 - u_4 v_2 - u_3 v_3 - u_4 v_3 - u_2 v_4 - u_3 v_4 - 2 u_4 v_4 = 0, \\
  L_2(u,v):= u_2 v_1 + u_1 v_2 + u_4 v_2 + u_3 v_3 + u_2 v_4 + u_4 v_4= 0, \\
  L_3(u,v):= u_3 v_1 + u_2 v_2 + u_4 v_2 + u_1 v_3 + u_3 v_3 + 2 u_4 v_3 + u_2 v_4 + 
    2 u_3 v_4 + 2 u_4 v_4 =0 , \\
 L_4(u,v):= u_4 v_1 + u_3 v_2 + u_4 v_2 + u_2 v_3 + u_3 v_3 + 2 u_4 v_3 + u_1 v_4 + 
    u_2 v_4 + 2 u_3 v_4 + 4 u_4 v_4 =0 . 
\end{array} \right.
\end{equation*} 
We show that if $u v \in 5 \mathcal{O}_K$, then either  $u_i = 0 \in \Z/5\Z$ for all $i$ or $v_i =0 \in \Z/5\Z$ for all $i$. We check this property by exhausting all the possibilities using computer algebra. 
For each $u, v \in (\Z/5\Z)^4 $, we compute $L_1(u,v), \ldots, L_4(u,v)$ modulo $5$. 
If every $L_i(u,v) $ is zero modulo $5$, we check whether all components of $u$ or $v$ are zero modulo $5$. If this does not happen, then there exist  two integers $u, v$ which do not belong to the ideal $5 \mathcal{O}_K$ whose product belongs to $5 \mathcal{O}_K$ and the ideal is not prime. Otherwise, it is prime.
 In the algorithm, we are reduced by symmetry to the case where $u_1 \leqslant v_1$ holds for a given order in $\Z/5\Z$ and the Mathematica code is given in the appendix below (see \ref{appendix_lem_5}).
\end{proof}


%
%


The ideal $\mathfrak{p}=(5)\mathcal{O}_F$ in $\mathcal{O}_F$ decomposes in $\mathcal{O}_K$ as $$\mathfrak{p}\mathcal{O}_K=\prod_{j=1}^g\mathfrak{P}_j^{e_j}=5\mathcal{O}_K,$$
since $5\mathcal{O}_K$ is a prime ideal. Because of this decomposition and the equation $[K:F]=2=e\cdot f\cdot g$, we get $f=2$ and $e=g=1$. Using the local-global principle for quaternion algebras with the extended $5$-adic valuation, we get the desired divisional quaternion algebra.

\begin{cor} The quaternion algebra $$B_1=\bigg(\frac{5,\,-2-2\sqrt{13}}{\mathbb{Q}(\sqrt{13})}\bigg)$$ with $I^2=5$ and $J^2=-2-2\sqrt{13}$ is a skew field with an order $\mathcal{M}=\mathcal{O}_K\oplus\mathcal{O}_K\cdot I$ and the map $'\colon B_1\longrightarrow B_1$ with $x'=J^{-1}\overline{x}J$ defines a positive anti-involution on $B_1$.
\end{cor}

Following the construction of chapter 9.4 in \cite{birkenhake_lange}, we get 

\begin{cor}
There exists a simple abelian variety $X$ of dimension $4$ whose endomorphism ring ${\rm End}(X)$ contains the order $\mathcal{M}$. The element $$\gamma=\dfrac{1 - \sqrt{13} + i \sqrt{2 + 2\sqrt{13}}}{4}$$ lies in $\mathcal{M}$ and defines an automorphism $f$ of $X$ whose analytic eigenvalues are the roots of $X^4-X^3-X^2-X+1$.
\\
The dynamical degrees of $f$ are $\lambda_0(f)=\lambda_4(f)=1$ and  $\lambda_1(f)=\lambda_2(f)=\lambda_3(f)=\lambda^2$.
\end{cor}

\section{Classification of automorphisms on low dimensional simple abelian varieties}\label{dim4}

\subsection{Lefschetz fixed-point formula}

To discuss the possible dynamical degrees of automorphisms on simple abelian varieties of dimension $1$ to $4$, we need precise insight in the possible eigenvalues of the analytic representation. Therefore, we again use the two fixed-point formulas as in the proof of Theorem \ref{thm_salem_tot_indef}. In this way, we will get information about the analytic eigenvalues by analyzing the algebraic structure of the possible endomorphism algebras. 
In the following, $g$ will be the dimension of the abelian variety $X$.
\\
First, we consider the case that $B:={\rm End}_\mathbb{Q}(X)$ is a field with $[B:\mathbb{Q}]=e$. For an endomorphism $f\in {\rm End}(X)$, we define $l:=[\mathbb{Q}(f):\mathbb{Q}]$ and $m:=[B:\mathbb{Q}(f)]$. We can compute the eigenvalues $\rho_1,\ldots,\rho_g,\overline{\rho}_1,\ldots,\overline{\rho}_g$ by \begin{equation}\label{FP1}\prod_{i=1}^g(X-\rho_i)(X-\overline{\rho}_i)=\#{\rm Fix}(f-X+1)=(\big(\prod_{j=1}^lX-\sigma_j(f)\big)^m)^{2g/e}.\end{equation}
\\
Secondly, we treat the case that $B:=\big(\frac{\alpha,\, \beta}{F}\big)$ is a quaternion algebra over the center $F$. If $B$ is a totally definite or indefinite quaternion algebra, then $F$ is totally real. If the abelian variety is of the second kind, then $F$ is a CM-field. An endomorphism $f\in{\rm End}(X)$ is of the form $a+bi+cj+dij$ with $a,b,c,d\in F$ and $i^2=\alpha, j^2=\beta$. As in the end of the proof of Theorem \ref{thm_salem_tot_indef}, the map $f$ corresponds to the numbers $t_{1/2}=a\pm\sqrt{b^2\alpha+c^2\beta-d^2\alpha\beta}$ and $F(t_{1/2})$ defines a quadratic extension of $F$. We define $e:=[F:\mathbb{Q}]$, $l:=[\mathbb{Q}(t_{1/2}):\mathbb{Q}]$ and $m:=[F(t_{1/2}):\mathbb{Q}(t_{1/2})]$, while $l \cdot m=2e$ holds. As in the proof of Theorem \ref{thm_salem_tot_indef}, the eigenvalues $\rho_1,\ldots,\rho_g,\overline{\rho}_1,\ldots,\overline{\rho}_g$ can be computed by \begin{equation}\label{FP2}\prod_{i=1}^g(X-\rho_i)(X-\overline{\rho}_i)=\#{\rm Fix}(f-X+1)=(\big(\prod_{j=1}^lX-\sigma_j(t_1)\big)^m\big(\prod_{j=1}^lX-\sigma_j(t_2)\big)^m)^{g/2e}.\end{equation}
\\
We briefly remember how to compute the dynamical degrees if the analytic eigenvalues are known:

We denote the eigenvalues of $\rho_r(f)\simeq\rho_a(f)\oplus\overline{\rho_a(f)}$ by $\rho_1,\ldots,\rho_g,\rho_{g+1}=\overline{\rho_1},\ldots,\rho_{2g}=\overline{\rho_g}$. The action $f^*$ of $f$ on ${\rm H}^1(X,\mathbb{C})$ is given by $^t\rho_a(f)\oplus\,^t\overline{\rho_a(f)}$ and ${\rm H}^n(X,\mathbb{C})=\wedge^n{\rm H}^1(X,\mathbb{C})$ holds. Then, the $k$-th dynamical degree $\lambda_k(f)$ is the product of the largest $2k$ pairwise distinct (in terms of the indices) eigenvalues $\rho_i$.
\\
The automorphisms of simple abelian varieties with trivial multiplication are exactly $\pm 1$ whose dynamical degrees $\lambda_i$ are always $1$. In the following statements, we concentrate on the non-trivial automorphisms. 

\subsection{Automorphisms of elliptic curves}

\begin{figure}[!ht]
\begin{itemize}
\centering
\item[(1)] \begin{tikzpicture}
[scale=.9, transform shape] \tikzstyle{every node} = [circle, fill=gray!30] \node (a) at (0, 0) {1};
\node (b) at +(180: 1) {0};
\foreach \from/\to in { a/b}
\draw [-] (\from) -- (\to);
\end{tikzpicture} $\quad\lambda_0=\lambda_1=1$
\caption{Dimension 1}
\end{itemize}
\end{figure}


\begin{table}[!ht]
\centering
\begin{tabular}{|c | c | c | c | c|}
\hline
\bf{multiplication} &  $\mathbb{Q}(f)$ & $[\mathbb{Q}(f):\mathbb{Q}]$ & diagram & Properties of $\lambda_j(f)$ \\
\hline 
trivial & $\mathbb{Q}$ & 1 & (1) & $1\in\mathbb{Z}$ \\
\hline
complex & CM-field & 2 & (1) & $1\in\mathbb{Z}$ \\
\hline
\end{tabular}
\caption{Dimension 1}
\end{table}

We have to determine the analytic eigenvalues that can be found on the left hand side of the equation \ref{FP1}, i.e $$(X-\rho_1)(X-\overline{\rho_1}).$$
\\
{\bf{Complex Multiplication:}}
 Let $f$ be a non-trivial automorphism on an elliptic curve $E$, i.e. ${\rm End}_\mathbb{Q}(E)$ is an imaginary quadratic number field and $f\in{\rm End}(E)\subseteq{\rm End}_\mathbb{Q}(E)$. The formula \ref{FP1} becomes in this setting $$(X-\rho_1)(X-\overline{\rho}_1)=(X-f)(X-\overline{f})=X^2-X(f+\overline{f})+1.$$
Because of $f\cdot\overline{f}=1$, the dynamical degrees are $\lambda_0=\lambda_1=1$ as claimed.

\begin{ex}\label{ellcurvex}
The polynomial $X^2+X+1$ is the minimal one of the algebraic integers $\frac{1\pm\sqrt{-3}}{2}$ of modulus $1$. The elliptic curve $\mathbb{C}/\Lambda$ with $\Lambda=\mathbb{Z}\oplus\frac{1+\sqrt{-3}}{2}\mathbb{Z}$, embedded in $\mathbb{P}_\mathbb{C}^2$ by the Weierstrass equation $x^3=y^2z+z^2y$, admits the automorphisms $\frac{1\pm\sqrt{-3}}{2}$.
\end{ex}

\subsection{Automorphisms of abelian surfaces} \label{section_surfaces}

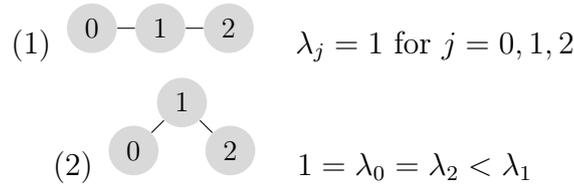
\begin{figure}[!h]
\begin{itemize}\label{surf}\centering
\item[(1)] \begin{tikzpicture}
[scale=.9, transform shape] \tikzstyle{every node} = [circle, fill=gray!30] \node (a) at (0, 0) {1};
\node (b) at +(180: 1) {0};
\node (c) at +(0: 1) {2};
\foreach \from/\to in { a/b, a/c}
\draw [-] (\from) -- (\to);
\end{tikzpicture} $\quad\lambda_j=1$ for $j=0,1,2$
\item[(2)] \begin{tikzpicture}
[scale=.9, transform shape] \tikzstyle{every node} = [circle, fill=gray!30] \node (a) at (0, 0) {1};
\node (b) at +(225: 1) {0};
\node (c) at +(315: 1) {2};
\foreach \from/\to in { a/b, a/c}
\draw [-] (\from) -- (\to);
\end{tikzpicture} $\quad1=\lambda_0=\lambda_2<\lambda_1$
\end{itemize}
\caption{Dimension 2}
\end{figure}

\bigskip 


\begin{table}[!ht]
\centering
\begin{tabular}{|c | c | c | c | c|}
\hline
 \bf{multiplication} &  $\mathbb{Q}(f)$ & $[\mathbb{Q}(f):\mathbb{Q}]$ & diagram & Properties of $\lambda_1(f)$\\
 \hline
 trivial & $\mathbb{Q}$ & 1 & (1) & $1\in\mathbb{Z}$ \\
 \hline
 real & totally real & 2 & (2) & constructible, Pisot \\
 & & & & deg. $2$ , tot. real \\
 \hline
 tot. indef. quaternion & totally real & 2 & (2) & constructible, Pisot \\
 & & & & deg. $2$, tot. real \\
 \hline
 - & CM-field & 2 & (1) & $1\in\mathbb{Z}$ \\
 \hline
 complex & totally real & 2 & (2) & constructible, Pisot,\\
 & & & & deg. $2$, tot. real \\
 \hline
-  & CM-field & 2 & (1) & $1\in\mathbb{Z}$ \\
  \hline
  & CM-field & 4 & (1), (2) & $1\in\mathbb{Z}$ or \\
-  & & & & constructible, Pisot,\\
  & & & & deg. $2.$, tot. real\\
  \hline
\end{tabular}
\caption{Dimension 2}
\end{table}

We use the equations \ref{FP1} and \ref{FP2} to determine the eigenvalues $\rho_i$ and further the dynamical degrees $\lambda_j$. The left hand side of the equations in dimension $2$ is $$\prod_{i=1}^2(X-\rho_i)(X-\overline{\rho}_i),$$ while the right hand side varies depending on the automorphism $f$. We prove the statements of table 2 by distinguishing the cases in the first column connected with the ones in the following two columns.
\\
{\bf{Real multiplication:}}
For the parameters in equation \ref{FP1}, we get $m=1$ and $g=e=l=2$, i.e. $$\prod_{i=1}^2(X-\rho_i)(X-\overline{\rho}_i)=\big(\prod_{j=1}^2X-\sigma_j(f)\big)^2.$$
This implies $\rho_1=\rho_3, \rho_2=\rho_4$ with $|\rho_1|>|\rho_2|$ and $\rho_1\cdot\rho_2=1$ which leads to the claim in table 2.
\\
\\
{\bf{Totally indefinite quaternion multiplication:}}
The parameters in equation \ref{FP2} become $e=m=1$ and $g=l=2$, i.e. $$\prod_{j=1}^2X-\sigma_j(t_1)\prod_{j=1}^2X-\sigma_j(t_2).$$
If $\mathbb{Q}(\sqrt{t_{1/2}})$ defines a real quadratic extension, then we get the same result as for real multiplication.
\\
If $\mathbb{Q}(\sqrt{t_{1/2}})$ defines an imaginary quadratic extension, then we get $\rho_2=\overline{\rho}_1=\rho_3$ and $\rho_4=\rho_1$ with $|\rho_1|=1$ and the dynamical degrees become as claimed.
\\
\\
{\bf{Complex multiplication:}}
First, the parameters $g=e=l=2$ and $m=1$ can occur, i.e. $$(\prod_{j=1}^2X-\sigma_j(f)\big)^2.$$
Since $\mathbb{Q}(f)$ must be an imaginary quadratic field, we get $\rho_2=\overline{\rho}_1=\rho_3$ and $\rho_4=\rho_1$ with $|\rho_1|=1$ which finishes this case.
\\
Secondly, the parameters can be $e=4$, $g=2$. 
\\For $l=m=2$, we are either in the previous case or in the same situation as in real multiplication.
\\For $m=1$ and $l=4$ we have $$\prod_{j=1}^4X-\sigma_j(f)$$ which offers two possibilities. First, one gets $\rho_1,\ldots,\rho_4$ that are all of absolute value $1$ which leads to diagram (1). Secondly, we have $\rho_3=\overline{\rho}_1, \rho_4=\overline{\rho}_2$ with $1\neq|\rho_1|=|\rho_2|^{-1}$ which implies diagram (2).

\subsection{Automorphisms of abelian threefolds}
\label{section_threefold}
\begin{figure}[!ht]
\begin{itemize}
\centering
\item[(1)] \begin{tikzpicture}
[scale=.9, transform shape] \tikzstyle{every node} = [circle, fill=gray!30] \node (a) at (0, 0) {2};
\node (b) at +(180: 1) {1};
\node (c) at +(180: 2) {0};
\node (d) at +(0: 1) {3};
\foreach \from/\to in { a/b, b/c, a/d}
\draw [-] (\from) -- (\to);
\end{tikzpicture} $\quad\lambda_j=1$ for $j=0,\ldots ,4$
\item[(2)] \begin{tikzpicture}
[scale=.9, transform shape] \tikzstyle{every node} = [circle, fill=gray!30] \node (a) at (0, 0) {2};
\node (b) at +(210: 1) {1};
\node (c) at +(230: 2) {0};
\node (d) at +(310: 2) {3};
\foreach \from/\to in { a/b, b/c, a/d}
\draw [-] (\from) -- (\to);
\end{tikzpicture} $\quad\lambda_1<\lambda_2<\lambda_1^2$
\item[(3)] \begin{tikzpicture}
[scale=.9, transform shape] \tikzstyle{every node} = [circle, fill=gray!30] \node (a) at (0, 0) {1};
\node (b) at +(230: 2) {0};
\node (c) at +(330: 1) {2};
\node (d) at +(310: 2) {3};
\foreach \from/\to in { a/b, a/c, c/d}
\draw [-] (\from) -- (\to);
\end{tikzpicture} $\quad\lambda_2^2>\lambda_1>\lambda_2$
\end{itemize}
\caption{Dimension 3}
\end{figure}
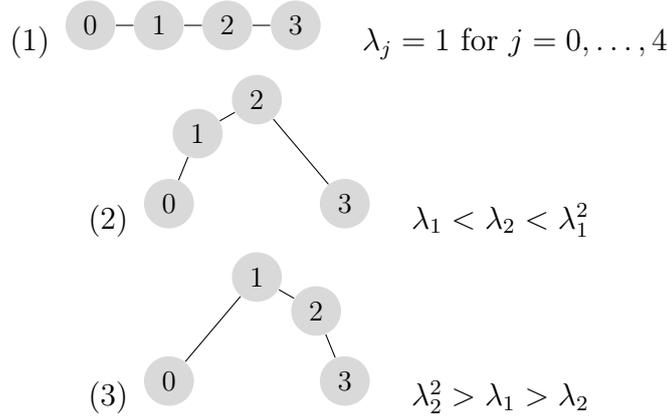

\begin{table}[!ht]
\centering
\begin{tabular}{|c | c | c | c | c|}
\hline
 \bf{multiplication} &  $\mathbb{Q}(f)$ & $[\mathbb{Q}(f):\mathbb{Q}]$ & diagram & Properties of $\lambda_1(f)$, $\lambda_2(f)$\\
 \hline
 trivial & $\mathbb{Q}$ & 1 & (1) & $1\in\mathbb{Z}$\\
 \hline
 real & totally real & 3 & (2), (3) & deg. $3$, tot. real \\
 \hline
 complex & totally real & 3 & (2), (3) & deg. $3$, tot. real \\
 \hline
-  & CM-field & 2 & (1) & $1\in\mathbb{Z}$\\
  \hline
-  & CM-field & 6 & (1), (2), (3) & $1\in\mathbb{Z}$ or \\
  & & & & deg. $3$, tot. real\\
  \hline
\end{tabular}
\caption{Dimension 3}
\end{table}

In this setting, the left hand side of the equations \ref{FP1} and \ref{FP2} becomes $$\prod_{i=1}^3(X-\rho_i)(X-\overline{\rho}_i).$$
\\
{\bf{Real multiplication:}} The possible parameters in equation \ref{FP1} are $m=1$ and $g=e=l=3$, i.e. $$P_f(X)=\big(\prod_{j=1}^3X-\sigma_j(f)\big)^2.$$ 
Therefore, we get the eigenvalues $\rho_{i+3}=\rho_i$ for $i=1,2,3$. The polynomial $X^3-5X+1$ has two roots greater than one and the polynomial $X^3-3X^2-2X+1$ has two roots smaller than one. These two polynomials are examples of the two possible distributions of eigenvalues and lead to the diagrams (2) and (3).
\\
\\
{\bf{Complex multiplication:}} First, the parameters can be $g=3$, $e=l=2$ and $m=1$, i.e. $$\big(\prod_{j=1}^2X-\sigma_j(f)\big)^3.$$ Here, $f$ has to lie in an imaginary quadratic field, such that $\rho_5=\rho_3=\rho_1$ and $\rho_6=\rho_4=\rho_2=\overline{\rho}_1$ which implies diagram (1).
\\
Secondly, the parameters can be $g=3$ and $e=6$. The first sub-case is  $m=3$ and $l=2$ which gives the same situation as in the first case. The second one is $l=3$ and $m=2$ which is the same as in real multiplication. The last sub-case, $l=6$ and $m=1$, defines a new situation, i.e. $$\prod_{j=1}^6X-\sigma_j(f)$$ with $\overline{\rho}_{i+1}=\rho_i$ for $i=1,2,3$.
\\
The polynomial $$X^6+X^4+X^3+2X^2+1$$ has roots with $|\rho_1|>|\rho_2|>1>|\rho_3|$ which implies diagram (2).
The roots of the polynomial $$X^6+X^4+X^3+1$$ have the properties $|\rho_1|>1>|\rho_2|>|\rho_3|$ and lead to diagram (3).

\begin{cor} The first dynamical degree $\lambda_1$ in diagram (3) in the case of real multiplication is a Pisot number. 
\end{cor}
\begin{proof} In real multiplication, diagram (3) occurs if one eigenvalue $\rho_1$ is a Pisot number. The first dynamical degree $\lambda_1$ is then given by $\rho_1^2$. The square of a Pisot number is again a Pisot number such that the corollary is proven.
\end{proof}

\subsection{Automorphisms of abelian fourfolds}
\label{section_fourfold}

\begin{figure}
\begin{itemize}
\centering
\item[(1)] \begin{tikzpicture}
[scale=.9, transform shape] \tikzstyle{every node} = [circle, fill=gray!30] \node (a) at (0, 0) {2};
\node (b) at +(180: 1) {1};
\node (c) at +(180: 2) {0};
\node (d) at +(0: 1) {3};
\node (e) at +(0: 2) {4};
\foreach \from/\to in { a/b, b/c, a/d, d/e}
\draw [-] (\from) -- (\to);
\end{tikzpicture} $\quad\lambda_j=1$ for $j=0,\ldots ,4$
\item[(2)] \begin{tikzpicture}
[scale=.9, transform shape] \tikzstyle{every node} = [circle, fill=gray!30] \node (a) at (0, 0) {2};
\node (b) at +(225: 1) {1};
\node (c) at +(225: 2) {0};
\node (d) at +(315: 1) {3};
\node (e) at +(315: 2) {4};
\foreach \from/\to in { a/b, b/c, a/d, d/e}
\draw [-] (\from) -- (\to); 
\end{tikzpicture} $\quad1<\lambda_1^2=\lambda_2=\lambda_3^2$

\item[(3)] \begin{tikzpicture}
[scale=.9, transform shape] \tikzstyle{every node} = [circle, fill=gray!30] \node (a) at (0, 0) {2};
\node (b) at +(210: 1) {1};
\node (c) at +(225: 2) {0};
\node (d) at +(345: 1) {3};
\node (e) at +(315: 2) {4};
\foreach \from/\to in { a/b, b/c, a/d, d/e}
\draw [-] (\from) -- (\to); 
\end{tikzpicture} $\quad\lambda_1<\lambda_2>\lambda_3,\, \lambda_1^2>\lambda_2<\lambda_3^2,\, \lambda_1<\lambda_3$

\item[(4)] \begin{tikzpicture}
[scale=.9, transform shape] \tikzstyle{every node} = [circle, fill=gray!30] \node (a) at (0, 0) {2};
\node (b) at +(210: 1) {1};
\node (c) at +(225: 2) {0};
\node (d) at +(330: 1) {3};
\node (e) at +(315: 2) {4};
\foreach \from/\to in { a/b, b/c, a/d, d/e}
\draw [-] (\from) -- (\to); 
\end{tikzpicture} $\quad\lambda_1<\lambda_2>\lambda_3,\, \lambda_1^2>\lambda_2<\lambda_3^2,\, \lambda_1=\lambda_3$

\item[(5)] \begin{tikzpicture}
[scale=.9, transform shape] \tikzstyle{every node} = [circle, fill=gray!30] \node (a) at (0, 0) {2};
\node (b) at +(195: 1) {1};
\node (c) at +(225: 2) {0};
\node (d) at +(330: 1) {3};
\node (e) at +(315: 2) {4};
\foreach \from/\to in { a/b, b/c, a/d, d/e}
\draw [-] (\from) -- (\to); 
\end{tikzpicture} $\quad\lambda_1<\lambda_2>\lambda_3,\, \lambda_1^2>\lambda_2<\lambda_3^2,\, \lambda_1>\lambda_3$

\item[(6)] \begin{tikzpicture}
[scale=.9, transform shape] \tikzstyle{every node} = [circle, fill=gray!30] \node (a) at (0, 0) {2};
\node (b) at +(215: 1) {1};
\node (c) at +(225: 2) {0};
\node (d) at +(15: 1) {3};
\node (e) at +(315: 2) {4};
\foreach \from/\to in { a/b, b/c, a/d, d/e}
\draw [-] (\from) -- (\to); 
\end{tikzpicture} $\quad\lambda_1<\lambda_2<\lambda_3,\, \lambda_1^2>\lambda_2,\, \lambda_2^2>\lambda_3$

\item[(7)] \begin{tikzpicture}
[scale=.9, transform shape] \tikzstyle{every node} = [circle, fill=gray!30] \node (a) at (0, 0) {2};
\node (b) at +(165: 1) {1};
\node (c) at +(225: 2) {0};
\node (d) at +(325: 1) {3};
\node (e) at +(315: 2) {4};
\foreach \from/\to in { a/b, b/c, a/d, d/e}
\draw [-] (\from) -- (\to); 
\end{tikzpicture} $\quad\lambda_1>\lambda_2>\lambda_3,\, \lambda_1<\lambda_2^2,\, \lambda_2<\lambda_3^2$

\item[(8)] \begin{tikzpicture}
[scale=.9, transform shape] \tikzstyle{every node} = [circle, fill=gray!30] \node (a) at (0, 0) {2};
\node (b) at +(225: 1) {1};
\node (c) at +(225: 2) {0};
\node (d) at +(15: 1) {3};
\node (e) at +(315: 2) {4};
\foreach \from/\to in { a/b, b/c, a/d, d/e}
\draw [-] (\from) -- (\to); 
\end{tikzpicture} $\quad\lambda_1<\lambda_2<\lambda_3,\, \lambda_1^2=\lambda_2,\, \lambda_2^2>\lambda_3$

\item[(9)] \begin{tikzpicture}
[scale=.9, transform shape] \tikzstyle{every node} = [circle, fill=gray!30] \node (a) at (0, 0) {2};
\node (b) at +(165: 1) {1};
\node (c) at +(225: 2) {0};
\node (d) at +(315: 1) {3};
\node (e) at +(315: 2) {4};
\foreach \from/\to in { a/b, b/c, a/d, d/e}
\draw [-] (\from) -- (\to); 
\end{tikzpicture} $\quad\lambda_1>\lambda_2>\lambda_3,\, \lambda_1<\lambda_2^2,\, \lambda_2=\lambda_3^2$

\item[(10)] \begin{tikzpicture}
[scale=.9, transform shape] \tikzstyle{every node} = [circle, fill=gray!30] \node (a) at (0, 0) {2};
\node (b) at +(225: 1) {1};
\node (c) at +(225: 2) {0};
\node (d) at +(335: 1) {3};
\node (e) at +(315: 2) {4};
\foreach \from/\to in { a/b, b/c, a/d, d/e}
\draw [-] (\from) -- (\to); 
\end{tikzpicture} $\quad\lambda_1<\lambda_2>\lambda_3,\, \lambda_1^2=\lambda_2<\lambda_3^2,\, \lambda_1<\lambda_3$

\item[(11)] \begin{tikzpicture}
[scale=.9, transform shape] \tikzstyle{every node} = [circle, fill=gray!30] \node (a) at (0, 0) {2};
\node (b) at +(205: 1) {1};
\node (c) at +(225: 2) {0};
\node (d) at +(315: 1) {3};
\node (e) at +(315: 2) {4};
\foreach \from/\to in { a/b, b/c, a/d, d/e}
\draw [-] (\from) -- (\to); 
\end{tikzpicture} $\quad\lambda_1<\lambda_2>\lambda_3,\, \lambda_1^2>\lambda_2=\lambda_3^2,\, \lambda_1>\lambda_3$
\end{itemize}
\end{figure}
\begin{figure}
\begin{itemize}\centering
\item[(12)] \begin{tikzpicture}
[scale=.9, transform shape] \tikzstyle{every node} = [circle, fill=gray!30] \node (a) at (0, 0) {2};
\node (b) at +(180: 1) {1};
\node (c) at +(225: 2) {0};
\node (d) at +(0: 1) {3};
\node (e) at +(315: 2) {4};
\foreach \from/\to in { a/b, b/c, a/d, d/e}
\draw [-] (\from) -- (\to); 
\end{tikzpicture} $\quad1<\lambda_1=\lambda_2=\lambda_3$
\end{itemize}
\caption{Dimension 4}
\end{figure}
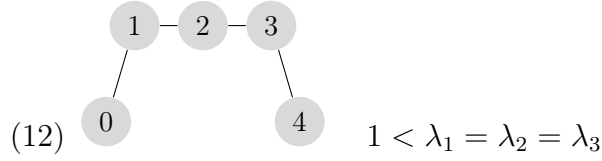

\begin{table}[ht]
\centering
\begin{tabular}{|c | c | c | c|c|}
\hline
 \bf{multiplication} &  $\mathbb{Q}(f)$ & $[\mathbb{Q}(f):\mathbb{Q}]$ & diagram & Properties of $\lambda_1(f)$ \\
 \hline
 trivial & $\mathbb{Q}$ & 1 & (1)& $\lambda_1(f)=1$ \\
 \hline
 real & totally real & 2 & (2)& Constructible \\
 \hline
- & totally real & 4 & (2) - (7) &  deg. $ \leqslant 4$ \\
 \hline
 tot. def. quaternion & totally real & 2 & (2)& Constructible \\
 \hline
- & CM-field & 2 & (1) & $\lambda_1(f)=1$ \\
 \hline
- & CM-field & 4 & (1), (2) & Constructible \\
 \hline
 tot. indef. quaternion & totally real & 2 & (2) & Constructible \\
 \hline
 -& totally real & 4 & (2) - (7) & deg. $\leqslant 4$ \\
 \hline
 -& CM-field & 2 & (1) & Constructible \\
 \hline
 -& CM-field & 4 & (1), (2) & Constructible \\
 \hline
- & real and complex & 4 & (6) - (12),& deg. $ \leqslant 4$ \\
 & embeddings & & &\\
 \hline
 second kind & totally real & 2 & (2) & Constructible \\
 \hline
-  & CM-field & 2 & (1) & $\lambda_1(f)=1$ \\
  \hline
-  & CM-field & 4 & (1), (2) & Constructible \\
  \hline
-  & totally real & 4 & (2) - (7) &  deg. $ \leqslant 4$ \\
 \hline
-  & CM-field & 8 & at most (2) - (7)&  deg. $ \leqslant 8$\\
  \hline
-  & quad. ext. of  & 4 & (1), (2) & Constructible \\
  & CM-field of deg. $2$ & &  &\\
  \hline
\end{tabular}
\caption{Dimension 4}
\end{table}
  
We use the equations \ref{FP1} and \ref{FP2} as in the previous subsections, while the left hand side in this case is always $$\prod_{i=1}^4(X-\rho_i)(X-\overline{\rho}_i).$$

{\bf{Real multiplication:}} First, the parameters of the equation \ref{FP1} can be $g=4$, $e=l=2$ and $m=1$, i.e. $$\big(\prod_{j=1}^2X-\sigma_j(f)\big)^4.$$
This gives us $\rho_1\neq\rho_2$ and $\rho_i=\rho_{i+2}$ which determines diagram (2).
\\
Secondly, the variables become $g=e=4$. 
\\
The first sub-case $l=m=2$ implies the same result as in the first case.
\\
The second sub-case $l=4$ and $m=1$, i.e. $$\big(\prod_{j=1}^4X-\sigma_j(f)\big)^2=P(X)^2,$$ provides the following possibilities. We list examples of all possible distributions of roots of the polynomial $P(X)$:
\begin{itemize}
\item $X^4-5X^2+1$ with $|\rho_1|=|\rho_2|>1>|\rho_3|=|\rho_4|=|\rho_1|^{-1}$ gives diagram (2).
\item $X^4+X^3-5X^2+X+1$ with the relations $|\rho_1|>|\rho_2|>1>|\rho_3|>|\rho_4|$ and $|\rho_2\cdot\rho_3|>1$ gives diagram (3).
\item $X^4+X^3-5X^2+X+1$ has roots with the relations $|\rho_1|>|\rho_2|>1>|\rho_3|>|\rho_4|$ with $|\rho_4|=|\rho_1|^{-1}$ and $|\rho_3|=|\rho_2|^{-1}$ which implies diagram (4).
\item $X^4-X^3-7X^2+1$ has roots with the relations $|\rho_1|>|\rho_2|>1>|\rho_3|>|\rho_4|$ and $|\rho_2\cdot\rho_3|<1$ such that we get diagram (5).
\item $X^4-X^3-139X^2+139X+1$ with $|\rho_1|>|\rho_2|>|\rho_3|>1>|\rho_4|$ determines diagram (6).
\item {\rm Pisot number:} $X^4-1133X^3-139X^2+13X+1$ with $|\rho_1|>1>|\rho_2|>|\rho_3|>|\rho_4|$ leads to diagram (7).
\end{itemize}

{\bf{Totally definite quaternion multiplication:}} If $f$ lies in the totally real number field $F$, then we are in the same situation as in the first case of the real multiplication. 
\\
In the following cases, the numbers $t_{1/2}=a\pm\sqrt{b^2\alpha+c^2\beta-d^2\alpha\beta}$ always define a totally imaginary quadratic extension of a totally real number field (either $F$ or $\mathbb{Q})$ and $t_1$ and $t_2$ are complex conjugate to each other (see \cite[Proposition 2.4]{herrig_19}).
\\
In the first new case, the parameters in equation \ref{FP2} become $g=4$, $l=2$ and $e=m=1$, i.e. $$\big(\prod_{j=1}^2X-\sigma_j(t_1)\prod_{j=1}^2X-\sigma_j(t_2)\big)^2.$$
We have $F=\mathbb{Q}$ and $t_{1/2}$ are imaginary quadratic numbers with $\overline{t_1}=t_2$ which means $|t_1|=|t_2|=1$ and $t_1^{-1}=t_2$. So, we are in diagram (1).
\\
The parameters $g=4$, $e=l=m=2$ lead to the formula $$\big(\prod_{j=1}^2X-\sigma_j(t_1)\big)^2\big(\prod_{j=1}^2X-\sigma_j(t_2)\big)^2$$ which gives the same possibilities for the eigenvalues as in the previous case.
\\
The last case deals with the values $g=4$, $l=4$, $e=2$, $m=1$, i.e. $$\prod_{j=1}^4X-\sigma_j(t_1)\prod_{j=1}^4X-\sigma_j(t_2).$$
The numbers $t_{1/2}$ can be roots of unity such that we get diagram (1). Otherwise, we have four roots $\rho_1,\ldots,\rho_4$ of the minimal polynomial with $\overline{\rho_1}=\rho_2$, $|\rho_1|=|\rho_2|>1$ and $\overline{\rho_3}=\rho_4$, $\rho_1^{-1}=\rho_3$ and $\rho_2^{-1}=\rho_4$. This defines diagram (2).
\\
\\
{\bf{Totally indefinite quaternion multiplication:}} As in the previous abstract, if $f$ lies in the totally real number field $F$, then we are in the same situation as in the first case of the real multiplication.
\\
The first difference to the definite case is that the numbers $t_{1/2}$ can define real or complex quadratic extensions of a totally real number field (again either $F$ or $\mathbb{Q}$).
\\
First, we consider the values $g=4$, $l=2$ and $e=m=1$, i.e. $$\big(\prod_{j=1}^2X-\sigma_j(t_1)\prod_{j=1}^2X-\sigma_j(t_2)\big)^2.$$
Since the roots of the minimal polynomial have to occur in conjugate pairs, we have either two complex roots of absolute value $1$ (diagram (1)) or two real roots of absolute value $\neq1$ (diagram (2)).
Further, the parameters $g=4$, $e=l=m=2$ imply $$\big(\prod_{j=1}^2X-\sigma_j(t_1)\big)^2\big(\prod_{j=1}^2X-\sigma_j(t_2)\big)^2$$ which leads to the same diagrams as in the previous setting.
\\
Secondly, the variables $g=l=4$, $e=2$ and $m=1$ give the formula $$\prod_{j=1}^4X-\sigma_j(t_1)\prod_{j=1}^4X-\sigma_j(t_2).$$
As a first sub-case we assume $t_{1/2}$ to be totally real numbers. We get $\rho_i=\rho_{i+4}=\overline{\rho_i}$ for totally real numbers such that $t_1$ and $t_2$ have the same minimal polynomial over $\mathbb{Q}$. Analogous to the real multiplication, we can realize the diagrams from (2) to (7).
\\
As a second sub-case, let $t_{1/2}$ define a CM-field over $F$. This determines exactly the same situation as in the last case of the totally definite quaternion multiplication, so we get the diagrams (1) and (2).
\\
The third sub-case is a new special one, the polynomial $\prod_{j=1}^4(X-\sigma_j(t_1))$ has two real and two complex roots. The existence of two real roots implies $\rho_i=\overline{\rho_i}$ for two analytic eigenvalues. Then, $\prod_{j=1}^4(X-\sigma_j(t_1))$ and $\prod_{j=1}^4(X-\sigma_j(t_2))$ have a common root such that these two minimal polynomials have to coincide.
\\
We first observe the situation that the complex roots are not of modulus $1$. The following list gives examples for every possible distribution of the analytic eigenvalues:
\begin{itemize}
\item The roots of the polynomial $$X^4+2X^3-3X^2+3X+1$$ are $|\rho_1|>|\rho_2|=|\rho_3|>1>|\rho_4|$ which gives diagram (6).
\item The polynomial $$X^4-4X^3+2X^2-3X+1$$ has the roots $|\rho_1|>1>|\rho_2|=|\rho_3|>|\rho_4|$ which leads to diagram (7).
\item The polynomial $$X^4-X^3+X^2+7X+1$$ with the roots $|\rho_1|=|\rho_2|>1>|\rho_3|>|\rho_4|$ implies diagram (8).
\item For the polynomial $$X^4-5X^3+3X^2-X+1$$ we get the roots $|\rho_1|>1>|\rho_2|>|\rho_3|=|\rho_4|$ and further diagram (9).
\item The polynomial $$X^4-X^3+3X^2+7X+1$$ has the roots $|\rho_1|=|\rho_2|>1>|\rho_3|=|\rho_4|$ and leads to diagram (10).
\item The polynomial $$X^4-5X^3+X^2+3X+1$$ comes with the roots $|\rho_1|>|\rho_2|>1>|\rho_3|=\rho_4|$ and implies diagram (11).
\end{itemize}

If one of the complex roots of $\prod_{j=1}^4(X-\sigma_j(t_1))$ is a Salem polynomial, then we get diagram (12). Its existence is provided by Theorem \ref{thm_salem_tot_indef}.
\\
\\
{\bf{Second kind and complex multiplication:}} We start with the cases that ${\rm End}_\mathbb{Q}(X)$ is a CM-field of degree $2, 4$ or $8$. If we take an element of the underlying totally number fields, we exactly rediscover all cases of the real multiplication. 
\\
Considering the values $g=4, e=l=2$ and $m=1$, we get $$(\prod_{j=1}^2X-\sigma_j(f))^{4}.$$
The two roots are imaginary quadratic and therefore lead to diagram (1). Taking $e=4$ with $l=m=2$ and $e=8$ with $l=2$ and $m=4$ gives the same result.
\\
The data $g=l=e=4$ and $m=1$ imply $$(\prod_{j=1}^4X-\sigma_j(f))^{2}.$$ and we have two possibilities: All roots lie on the unit circle (diagram (1)) or two outside and two inside the unit circle (diagram (2)).
The parameters $e=8$, $l=4$ and $m=2$ give the same result.
\\
The variables $g=4$, $e=l=8$ and $m=1$ give $$\prod_{j=1}^8X-\sigma_j(f).$$
The roots of this polynomial have to occur in conjugate pairs, such that we have the following possibilities for the 4 conjugate pairs: One pair lies inside and three pairs outside the unit circle, two inside and two outside or three inside and one outside. In the middle case, the two pairs outside the unit circle might be of the same or different absolute value and therefore the two inside. All these possibilities would at most lead to the diagrams (2) - (7) that already occur for elements of the underlying totally real subfield.
\\
As the last and new case, given by the classification of the endomorphism algebra, we consider a division algebra of dimension $4$ over a CM-field of degree $2$. By  \cite[Proposition 1.2.1]{gille_szamuely}, the algebra $B$ is a divisional quaternion algebra over an imaginary quadratic number field $F$. Further, by \cite[Corollary 2.2.10]{gille_szamuely}, the quadratic extension $H:=F(t_{1/2})$ splits the quaternion algebra $B$. The cases that the endomorphism $f$ lies in the center $F$ or in a real quadratic extension of $F$ were already considered.
\\
Let $\sigma_j$ and $\sigma'_j$ be the $\mathbb{Q}$-embeddings of $t_1$ and $t_2$ in the algebraic closure $\mathbb{C}$, then we have for $g=4$ and $e=2$ the equation
$$\prod_{j=1}^4X-\sigma_j(t_1)\prod_{j=1}^4X-\sigma'_j(t_2).$$
\\
The number field $F$ is of the Form $\mathbb{Q}(\sqrt{-d})$ with $d\in\mathbb{Z}_{>0}$ squarefree. If we take an element $\tau$ of the number field $H$, then the four roots of its minimal polynomial are either all complex or two are real and the other two are complex. The elements $t_1$ and $t_2$ define the same field extension $H$ of $F$, the images of $\sigma_j$ and $\sigma'_j$ are therefore either all complex or two of each four are real and two are complex. In the second situation, the minimal polynomials of $t_1$ and $t_2$ have to be the same, since the real roots have to be the same because of $\rho_{i+4}=\overline{\rho_i}$ for the analytic eigenvalues.
\\
We start with the case of two complex and two real embeddings of $t_1$ or $t_2$. By definition, the number field $H$ is imaginary, but not totally complex. 
\\
Let us first assume that $t_1=a+\sqrt{t}$ is real. We can write $a$ as $a_1+a_2\sqrt{-d}$ with $a_1,a_2\in \mathbb{Q}$. If $a_2$ is unequal $0$, then the equation $a_1+a_2\sqrt{-d}+\sqrt{t}\in\mathbb{R}$ implies $\sqrt{t}=-a_2\sqrt{-d}$. This is a contradiction, since $\sqrt{t}$ defines a field extension of $F=\mathbb{Q}(\sqrt{-d})$.
\\
If $a_2=0$, then we have $t\in\mathbb{Q}_{>0}$ and $H=\mathbb{Q}(\sqrt{-d},\sqrt{t})$ is  a CM-field. This is a contradiction, since $H$ is not totally complex.
\\
We now assume $t_1$ to be complex. The elements $t_1$ and $t_2$ define the same field extension of $F$ and have the same minimal polynomial. By the argument of the previous abstract, $t_2$ also has to be complex. Further, $t_1$ and $t_2$ are different and the only complex roots of the same minimal polynomial which means $\overline{t_1}=t_2$ must hold. Hence, the field $H$ is closed under complex conjugation as an imaginary field that is not totally complex. Especially, $H$ is not a CM-field.
\\
By \cite[Theorem 2]{daileda}, the field $H$ must contain unimodular units that are not roots of unity. These are exactly complex roots of a Salem polynomial of degree $4$. Let $\tau$ be a complex root of such a Salem polynomial $P$. Then, we have $\tau+\overline{\tau}=\sqrt{e}\in H$ with $e\in\mathbb{Z}_{>0}$ square-free. The elements $\sqrt{e}$ and $\sqrt{-d}$ are contained in $H$ and $H$ is of degree $4$ over $\mathbb{Q}$ such that $H=\mathbb{Q}(\sqrt{-d},\sqrt{e})$ is a CM-field, a contradiction.
\\
\\
We come to the case that all roots of the minimal polynomials of $t_1$ and $t_2$ are complex. We start with the case that $t_1$ is of absolute value $1$. Then, $t_1$ has to be a primitive $4$-th root of unity. These are exactly the numbers $$\pm\frac{\sqrt{2}}{2}\pm\frac{\sqrt{-2}}{2},\quad \pm\frac{\sqrt{3}}{2}\pm\frac{\sqrt{-1}}{2}, \quad -\frac{1}{4}+\frac{\sqrt{5}}{4}\pm\sqrt{\frac{5}{8}+\frac{\sqrt{5}}{8}}$$ 
$${\rm and}\quad -\frac{1}{4}-\frac{\sqrt{5}}{4}\pm\sqrt{\frac{5}{8}-\frac{\sqrt{5}}{8}}.$$
By the form of $t_1$, we see that if $t_1$ is one of these numbers, then $t_2$ is one of them, too. The last two numbers with the double root are obviously not possible. Hence, if $t_1$ is of modulus $1$, then all analytic eigenvalues are roots of unity and we get diagram (1).
\\
We assume that $t_1$ and $t_2$ are not of modulus $1$. The general roots of a polynomial of degree $4$ with root $t_1=a+\sqrt{t}=a_1+a_2\sqrt{-d}+\sqrt{b_1+b_2\sqrt{-d}}$ are $$a_1\pm a_2\sqrt{-d}+\sqrt{b_1\pm b_2\sqrt{-d}}\quad {\rm and} \quad a_1\pm a_2\sqrt{-d}-\sqrt{b_1\pm b_2\sqrt{-d}}$$ with $a_1,a_2,b_1,b_2\in\mathbb{Q}$. This implies that $t_1$ and $t_2$ have the same minimal polynomial and we get the analytic eigenvalues $|\rho_1|=|\rho_2|=|\rho_3|=|\rho_4|>1>|\rho_5|=|\rho_6|=|\rho_7|=|\rho_8|$ leading to diagram (2).

\section{Dynamical results}

\subsection{Hyperbolicity}

In this section, we view an abelian variety $X$ as the underlying smooth compact manifold corresponding to the complex torus $\C^g/\Lambda$. 
The dynamical properties of an endomorphism can be read directly from the properties of the action on the universal cover, which is a linear map of $\C^g$ which preserves the lattice $\Lambda$. 
In this particular setting, some important dynamical properties are well understood. 
Recall that a linear map $L \colon \C^g \longrightarrow \C^g$ is \textit{hyperbolic} if the vector space $\C^g$ can be decomposed as:
\begin{equation*}
\C^g = E_s \oplus E_u,
\end{equation*} 
where $E_s$ is a subspace containing eigenspaces associated to eigenvalues $\lambda$ such that $|\lambda| < 1$ and $E_u$ is a subspace containing eigenspaces associated to eigenvalues $\lambda$ satisfying $|\lambda|>1$.

An {endomorphism $f\colon \C^g/\Lambda \longrightarrow \C^g/\Lambda$ of an abelian variety is called \textit{hyperbolic}} if its lift  to the universal cover $\C^g$ is hyperbolic.
In other words, one can read the hyperbolicity of an endomorphism on the sequence of dynamical degrees.

\begin{prop} Let $f\colon\C^g/\Lambda \longrightarrow\C^g/\Lambda$ be an endomorphism of an abelian variety of complex dimension $g$, then the following conditions are equivalent.
\begin{enumerate}
\item[(i)] The map $f$ is hyperbolic.
\item[(ii)] The analytic representation of $f$ has no eigenvalue on the unit circle.
\item[(iii)] For any $k\leq g-1$, we have $\lambda_k(f) \neq \lambda_{k+1}(f)$. 
\end{enumerate}
\end{prop}

We also obtain the following convergence of forms through a straightforward calculation. 

\begin{prop}
\label{prop_hyperbolic}
 Let $f \colon \C^g / \Lambda \longrightarrow \C^g / \Lambda$ be an automorphism which is hyperbolic such that $\lambda_k(f) = \max_{j \leqslant g} \lambda_j(f)$. 
Then there exists a unique (up to scaling) $(k, k)$-form $\alpha$ with constant coefficients  such that for any $(k,k)$-form $\Omega$, the sequence
\begin{equation}
\dfrac{1}{N} \sum_{j\leqslant N} \dfrac{1}{\lambda_k(f)^j} (f^j)^* \Omega
\end{equation}
converges to a multiple of $\alpha$. 
\end{prop}

\subsection{Partial hyperbolic automorphisms}

We now discuss the non-hyperbolic automorphisms. 
To do so, if $f$ is an endomorphism of $\C^g/\Lambda$, we fix a basis $dz_1 , \ldots , dz_g, d\bar z_1, \ldots , d\bar z_g$ of $\C^g$ diagonalizing the linear transformation on the tangent space. 
For simplicity, let us order the eigenvalues and eigenvectors in a decreasing order.

\begin{thm} \label{thm_partial} Let $f\colon\C^g/\Lambda \longrightarrow \C^g/\Lambda$ be an automorphism of an abelian variety of complex dimension $g$.  
Suppose that the spectral radius of the linear map associated to $f$ is a Salem number. 
 Then,  for any integer $1\leqslant k \leqslant g-1$ and any  $(k,k)$-form $\Omega$, the sequence
$$  \dfrac{1}{N} \sum_{j\leqslant N} \dfrac{1}{\lambda_k(f)^j}(f^j)^* \Omega $$
converges to the constant form:
\begin{equation*}
\sum_{\substack{|I|=k \\
1\in I}} a_I dz_I\wedge d\bar z_I,
\end{equation*}
where  $$a_I = \int \Omega \wedge dz_{I^c} \wedge d\bar z_{I^c} .$$
\end{thm}

\begin{rem} The above proposition shows that the averages in the space of forms realize a projection onto a finite dimensional vector space of dimension $g-2 \choose k-1$.
\end{rem}

\begin{proof}
Suppose first $k = 1$. Let us first observe that the eigenvalue $\lambda_1(f)$ is the spectral radius of the action $f^*$ on $\Lambda^{1,1}(\C^g)$.
We order the eigenvalues of $f$ such that $ |\rho_1|  > |\rho_2| = \ldots = |\rho_{g-1}| =1  > |\rho_g|$.
Let us choose a basis of $(\C^g)^*$ diagonalizing the linear map $f$ so that:
\begin{equation*}
f^* z_i  = \rho_i z_i, \ f^* \bar z_i = \bar \rho_i \bar z_i, 
\end{equation*}
for $i =1$ or $i= g$
and such that
\begin{equation*}
f^* z_i = e^{i\theta_i} z_i,\ f^* \bar z_i = e^{-i\theta_i} \bar z_i,
\end{equation*}
for $2 \leqslant i \leqslant  g-1$. 
Now the eigenvalue $\lambda_1(f)$ is a simple eigenvalue for the action of $f^* $ on $\Lambda^{1,1}(\C^g)$ and the eigenvector associated to $\lambda_1(f)$ is given by:
\begin{equation*}
\alpha=dz_1 \wedge d\bar z_1.
\end{equation*}
Let us now take a general  strictly positive $(1, 1 ) $-form:
\begin{equation*}
\omega = \sum_{i} \omega_{i,j}\, dz_i \wedge d\bar z_j,
\end{equation*}
where the $\omega_{i,j}$ are functions.
Taking the pullback $(f^j)^* \omega / \lambda_1(f)^j$, the only term that remains is of the form:
\begin{equation*}
  \dfrac{1}{\lambda_1(f)^j}(f^j)^* \omega = (\omega_{1,1} \circ f^j) \alpha  + o (1).
\end{equation*} 
The average is thus given by:
\begin{equation*}
\dfrac{1}{N}  \sum_{j\leqslant N}\dfrac{1}{\lambda_1(f)^j}(f^j)^* \omega = \dfrac{1}{N}\sum_{j \leqslant N} (\omega_{1,1} \circ f^j)   + o(1).
\end{equation*}
Since the linear transformation $L$  has no eigenvalues that are roots of unity, the transformation $f$ is ergodic by \cite[Corollary 1.10.1]{walters} and Birkhoff's ergodic theorem (see e.g \cite[Theorem 1.14]{walters}) yields:
\begin{equation*}
\lim_{N\rightarrow +\infty }\dfrac{1}{N}  \sum_{j\leqslant N}\dfrac{1}{\lambda_1(f)^j}(f^j)^* \omega = \left (\int_X \omega_{1,1} d\mu \right ) \alpha, 
\end{equation*}
where $\mu$ is the Haar measure on the abelian variety $X$.

\medskip

Let us now suppose that $k \geqslant 2$ and $k \leqslant g-1$. 
Take a pure $(k,k)$-form:
\begin{equation*}
dz_I \wedge d\bar z_J,
\end{equation*} 
where the multi-indices $I$, $J$ are of length $k$. 
Observe that $1/\lambda_k(f)^m(f^m)^* dz_I \wedge d\bar z_J $ converges with exponential speed to zero whenever  $I$ or $J$ do not contain $ 1$. 
This yields the following estimate:
\begin{equation*}
\dfrac{1}{\lambda_k(f)^m}(f^m)^* dz_I \wedge d\bar z_J  = e^{i m( \theta_{I'} - \theta_{J'})} dz_I \wedge d\bar z_J,
\end{equation*}
where $I = I' \cup \{1\}$, $J = J' \cup \{1 \}$ and $I',J'$ are contained in $\{ 2 , \ldots ,g-1 \}$, and where $\theta_{I'}, \theta_{J'}$ are given by:
\begin{equation*}
\theta_{I'} = \sum_{i\in I'} \theta_i. 
\end{equation*}
We obtain the asymptotic relation:
\begin{equation*}
\dfrac{1}{N}\sum_{m\leqslant N} \dfrac{1}{\lambda_k(f)^m} (f^m)^* \omega = \dfrac{1}{N}\sum_{m\leqslant N}  \sum  \left (\omega_{I,J} \circ f^m \right ) e^{i m (\theta_{I'} - \theta_{J'})} dz_{I} \wedge d\bar z_{J}  + o(1),
\end{equation*}
where the sum is taken over all multi-indices of the form $I = I' \cup \{1\}, J= J' \cup \{1 \}$ where $I', J' \subset \{ 2 , \ldots, g-1 \}$.

Consider the operator $T_{I,J}$ on $L^2(\mu, \mathbb{C})$ given by:
\begin{equation*}
T_{I,J} = e^{i (\theta_I - \theta_J)} f^*. 
\end{equation*}
By construction, the operator $T_{I', J'}$ is unitary for $I', J' \subset \{2 , \ldots , n-1\}$.
By Von-Neumann's ergodic theorem \cite[Corollary 1.14.1]{walters}, the sequence of functions:
\begin{equation*}
\dfrac{1}{N} \sum_{n \leqslant N} T_{I',J'}^n (\omega_{I,J}) 
\end{equation*}
converges to a function $\bar \omega \in L^2(\mu, \mathbb{C})$ satisfying:
\begin{equation*}
T_{I',J'} \bar \omega = \bar \omega.  
\end{equation*}
We thus obtain:
\begin{equation*}
e^{i(\theta_{I'} - \theta_{J'})} \bar \omega \circ f = \bar \omega.  
\end{equation*}
Hence, we get:
\begin{equation} \label{eq_fourier}
f^* \bar \omega = e^{i (\theta_{J'} - \theta_{I'})} \bar \omega. 
\end{equation}
Taking the Birkhoff average and applying Birkhoff ergodic theorem, we have:
\begin{equation*}
 \lim_{N \rightarrow +\infty} \dfrac{1}{N}  \sum_{n\leqslant N}(f^n)^* \bar \omega = \int \bar \omega d\mu = \lim_{N\rightarrow +\infty} \dfrac{1}{N} \sum_{n \leqslant N} e^{i n(\theta_{J'} - \theta_{I'})} \bar \omega = 0,
\end{equation*}
under the condition that $I' \neq J'$.
This shows that the constant term in the Fourier decomposition of $\bar \omega$ is zero. Let us show now that $\bar \omega$ is zero.
Take $l \in \Lambda^*$ and denote by $c_l(\bar \omega)$ the Fourier coefficient of the function $\bar \omega$ associated to $l$. 
\eqref{eq_fourier} shows that:
\begin{equation*}
c_{l} (f^*\bar \omega) = e^{i(\theta_{I'} - \theta_{J'})} c_l(\bar \omega). 
\end{equation*}
Moreover, if $A$ is the  matrix associated to $f$ acting on the dual lattice $\Lambda^*$, then:
\begin{equation*}
c_l(f^*\bar\omega) = c_{A\cdot l} (\bar\omega) = e^{i(\theta_{I'} - \theta_{J'})} c_l(\bar \omega).
\end{equation*}
Iterating $n$ times, one obtains the relation:
\begin{equation*}
 c_{A^n\cdot l} (\bar\omega) =  e^{i n(\theta_{I'} - \theta_{J'})} c_l(\bar \omega).
\end{equation*}
Since the sequence $A^n \cdot l $ diverges to infinity as $A$ has an eigenvalue strictly larger than $1$ and since the Fourier coefficient of $\bar \omega$ decreases exponentially fast at infinity, we deduce that $c_l(\bar \omega)$ also decreases exponentially fast, hence they are all zero. 
This shows that $\bar \omega = 0$ and we have obtained that:
\begin{equation*}
 \dfrac{1}{N}\sum_{m\leqslant N}   \left (\omega_{I,J} \circ f^m \right ) e^{i m (\theta_{I'} - \theta_{J'})} = 0,
\end{equation*}
whenever $I' , J' \subset \{2, \ldots, g-1 \} $ are distinct multi-indices. 
The asymptotic formula can be simplified as:
\begin{equation*}
\dfrac{1}{N}\sum_{m\leqslant N} \dfrac{1}{\lambda_k(f)^m} (f^m)^* \omega = \dfrac{1}{N}\sum_{m\leqslant N}  \sum_{\substack{|I|=k \\
1\in I}}  \left (\omega_{I,I} \circ f^m \right )  dz_{I} \wedge d\bar z_{I}  + o(1).
\end{equation*}
 Birkhoff ergodic theorem thus yields:
 \begin{equation*}
 \lim_{N\rightarrow +\infty} \dfrac{1}{N} \sum_{m\leqslant N} \dfrac{1}{\lambda_k(f)^m} (f^m)^*\omega =  {\sum_{\substack{|I|=k \\
1\in I}} } \left ( \int \omega_{I,I} \right ) dz_I \wedge d\bar z_I. 
 \end{equation*}
The proposition then follows by identifying the integral with each coefficient $a_I$. 
%
%
%
\end{proof}

\subsection{Proof of Theorem A}

We start by the following simple criterion for an endomorphism to be integrable.

\begin{prop} \label{prop_integrability}
Let $f\colon X\longrightarrow X$ be an automorphism of a $g$-dimensional abelian variety. If $1$ is an eigenvalue of the analytic representation $\rho_a(f)$ of algebraic multiplicity $1\leq \mu_{\rm alg}(1)=h <g$, then $X$ is isogenous to the product of abelian varieties $Y_1\times Y_2$, where $Y_1$ is of dimension $h$.
\\
In particular, $Y_1$ is an abelian subvariety of $X$ fixed by the automorphism $f$.
\end{prop}

We postpone the proof of the proposition above. Let us show Theorem A when $X$ is of dimension $1, 2, 3$ and $4$. 

Take an automorphism $f\colon X \longrightarrow X$ of positive entropy.

%

Suppose that $X$ is an abelian surface. Then either $X$ is simple or non-simple. If $X$ is simple, then our classification in Section \ref{section_surfaces} shows that any automorphism of positive entropy is hyperbolic. Theorem A then follows from Proposition \ref{prop_hyperbolic}. 
In the case where $X$ is not simple, we claim that no roots of unity are eigenvalues of the analytic representation. Indeed, if it were the case, Proposition \ref{prop_integrability} shows that $f$ would be semi-conjugate to an automorphism on an elliptic curve, but then $f$ would have zero entropy. 
Moreover, if the analytic representation has an eigenvalue that is on the unit circle but not a root of unity, then the other eigenvalue is its conjugate and $f$ also has zero entropy.

\smallskip 

Suppose now that $X$ is an abelian threefold. If $X$ is simple, then our classification in Section \ref{section_threefold} shows that $f$ is hyperbolic. And we conclude using Proposition \ref{prop_hyperbolic}.
If $X$ is not simple, then two cases appear. If the analytic representation has an eigenvalue that is a root of unity, then $f$ is semi-conjugate to an automorphism on an elliptic curve or an abelian surface. 
Otherwise, no eigenvalues of the analytic representation is a root of unity, this implies that the three eigenvalues of the analytic representation are also not on the unit circle because $f$ has positive entropy (as their product is $1$). Hence $f$ is hyperbolic and we can apply Proposition \ref{prop_hyperbolic}.

\smallskip

In the case where $X$ is a fourfold. If $X$ is simple, then our classification in Section \ref{section_fourfold} shows that $f$ is either hyperbolic or the first dynamical degree is the square of a Salem number. In either situation, we apply Proposition \ref{prop_hyperbolic} or Theorem  \ref{thm_partial} respectively and obtain Theorem A.
If $X$ is not simple, three cases appear. If the analytic representation has no eigenvalue on the unit circle, then $f$ is hyperbolic and Theorem A holds by Proposition \ref{prop_hyperbolic}. 
If the analytic representation has a root of unity as an eigenvalue, then Proposition \ref{prop_integrability} shows that $f$ is semi-conjugate to an automorphism on an abelian variety of smaller dimension. 
If the analytic representation has an eigenvalue on the unit circle that is not a root of unity, then its complex conjugate is another eigenvalue and the other two eigenvalues are in the unit disk and on the outside respectively. We obtain that $\lambda_1(f)$ is the square of a Salem number and we conclude using Theorem \ref{thm_partial}.

\begin{proof}[Proof of Proposition \ref{prop_integrability}]
Let $1$ be an analytic eigenvalue of the automorphism $f\colon X\longrightarrow X$ of the abelian variety $\mathbb{C}^g/\Lambda$. We consider the subtorus $Y_1:={\rm ker}(f-{\rm id})_0$, the connected component of ${\rm ker}(f)$ that contains $0$ (see \cite[Proposition 1.2.4]{birkenhake_lange}). It must be of positive dimension. Because, otherwise, $f-{\rm id}$ would be an isogeny and $\rho_a(f)-\mathbf{1}$ an invertible matrix which is impossible, since $1$ is an eigenvalue of $\rho_a(f)$. 
\\
We determine the dimension of $Y_1$: It is known that $(\rho_a(f)-\mathbf{1})^{-1}(\mathbb{C}^g)_0$ has to be a subspace $V$ of $\mathbb{C}^g$ of dimension $h$. Further, we know that $\Gamma=(\rho_a(f)-\mathbf{1})^{-1}(\mathbb{C}^g)_0\cap\Lambda$ is a lattice in $V$ such that $Y_1=V/\Gamma$ becomes a complex subtorus of $X$ of dimension $h$. Since complex subtori of abelian varieties are abelian varieties, $Y_1$ is an abelian subvariety of $X$.
\\
By Poincare's reducibility theorem \cite[Theorem 5.3.7]{birkenhake_lange}, we get an abelian subvariety $Y_2$ of $X$ such that $X$ is isogenous to $Y_1\times Y_2$.
\\
By construction, we have $(f-{\rm id})(Y_1)=0$, i.e. $f(Y_1)=Y_1$.
\end{proof}

%
%
%
%
%
%
%
%
%
%
%
%
%
%
%
%

\section*{Appendix -- Mathematica code for Lemma \ref{lem_ideal_5_technical}}
\label{appendix_lem_5}
\addcontentsline{toc}{section}{Appendix -- Mathematica code for Lemma \ref{lem_ideal_5_technical}}

The code we used for the proof of the lemma is as follows. 
The function $isprime$ returns True if the ideal $5 \mathcal{O}_K$ is prime and False otherwise.

\begin{doublespace}
\noindent\(\pmb{\text{}}\\
\pmb{\text{L1}[\text{x$\_$}, \text{y$\_$}] \text{:=} }\\
\pmb{\text{Mod}[\text{Expand}[\text{u1} \text{v1}-\text{u4} \text{v2}-\text{u3} \text{v3}-\text{u4} \text{v3}-\text{u2} \text{v4}-\text{u3} \text{v4}-2
\text{u4} \text{v4} \text{/.} }\\
\pmb{\{\text{u1} \to  x[[1]], \text{u2} \to  x[[2]], \text{u3}\to  x[[3]], \text{u4}\to  x[[4]],\text{v1} \to  y[[1]], }\\
\pmb{\text{v2} \to  y[[2]], \text{v3}\to  y[[3]], \text{v4}\to  y[[4]] \}],5]}\\
\pmb{\text{L2}[\text{x$\_$}, \text{y$\_$} ] \text{:=} }\\
\pmb{\text{Mod}[\text{Expand}[(\text{u2} \text{v1}+\text{u1} \text{v2}+\text{u4} \text{v2}+\text{u3} \text{v3}+\text{u2} \text{v4}+\text{u4} \text{v4})
\text{/.} }\\
\pmb{\{\text{u1} \to  x[[1]], \text{u2} \to  x[[2]], \text{u3}\to  x[[3]], \text{u4}\to  x[[4]],\text{v1} \to  y[[1]], }\\
\pmb{\text{v2} \to  y[[2]], \text{v3}\to  y[[3]], \text{v4}\to  y[[4]] \}],5]}\\
\pmb{\text{L3}[\text{x$\_$},\text{y$\_$}] \text{:=} }\\
\pmb{\text{Mod}[\text{Expand}[(\text{u3} \text{v1}+\text{u2} \text{v2}+\text{u4} \text{v2}+\text{u1} \text{v3}+\text{u3} \text{v3}+2 \text{u4} \text{v3}+\text{u2}
\text{v4}+2 \text{u3} \text{v4}+2 \text{u4} \text{v4}) \text{/.} }\\
\pmb{\{\text{u1} \to  x[[1]], \text{u2} \to  x[[2]], \text{u3}\to  x[[3]], \text{u4}\to  x[[4]],\text{v1} \to  y[[1]], }\\
\pmb{\text{v2} \to  y[[2]], \text{v3}\to  y[[3]], \text{v4}\to  y[[4]] \}],5]}\\
\pmb{\text{L4}[\text{x$\_$}, \text{y$\_$}] \text{:=}\text{  }}\\
\pmb{\text{Mod}[\text{Expand}[(\text{u4} \text{v1}+\text{u3} \text{v2}+\text{u4} \text{v2}+\text{u2} \text{v3}+\text{u3} \text{v3}+2 \text{u4} \text{v3}+\text{u1}
\text{v4}+\text{u2} \text{v4}+2 \text{u3} \text{v4}+4 \text{u4} \text{v4}) \text{/.} }\\
\pmb{\{\text{u1} \to  x[[1]], \text{u2} \to  x[[2]], \text{u3}\to  x[[3]], \text{u4}\to  x[[4]],\text{v1} \to  y[[1]], }\\
\pmb{\text{v2} \to  y[[2]], \text{v3}\to  y[[3]], \text{v4}\to  y[[4]] \}],5]}\\
\pmb{\text{isprime}\text{:=} \text{Module}[\{ \text{valueisprime}, \text{i1}, \text{j1} , \text{i2} , \text{j2} , \text{i3}, \text{j3}, \text{i4},
\text{j4}\}, }\\
\pmb{\text{valueisprime}\text{:=}\text{True}; }\\
\pmb{\text{For}[\text{i1}=0, \text{i1}< 5 , \text{i1}\text{++},\text{  }}\\
\pmb{\text{For}[\text{i2}=0, \text{i2}< 5 , \text{i2}\text{++},\text{  }}\\
\pmb{ \text{For}[\text{i3}=0, \text{i3}< 5 , \text{i3}\text{++},\text{  }}\\
\pmb{ \text{For}[\text{i4}=0, \text{i4}< 5 , \text{i4}\text{++},\text{  }}\\
\pmb{ \text{For}[\text{j1}=0, \text{j1}\text{$<$=} \text{i1} , \text{j1}\text{++},\text{  }}\\
\pmb{ \text{For}[\text{j2}=0, \text{j2}< 5 , \text{j2} \text{++},\text{  }}\\
\pmb{\text{For}[\text{j3}=0, \text{j3}< 5 , \text{j3}\text{++},\text{  }}\\
\pmb{\text{For}[\text{j4}=0, \text{j4}< 5 , \text{j4}\text{++},\text{  }}\\
\pmb{\text{If} [ \text{L1}[\{\text{i1},\text{i2},\text{i3},\text{i4}\} , \{\text{j1},\text{j2},\text{j3},\text{j4}\}] == }\\
\pmb{\text{L2}[\{\text{i1},\text{i2},\text{i3},\text{i4}\} , \{\text{j1},\text{j2},\text{j3},\text{j4}\}] == }\\
\pmb{\text{L3}[\{\text{i1},\text{i2},\text{i3},\text{i4}\} , \{\text{j1},\text{j2},\text{j3},\text{j4}\}] \text{==}}\\
\pmb{\text{L4}[\{\text{i1},\text{i2},\text{i3},\text{i4}\} , \{\text{j1},\text{j2},\text{j3},\text{j4}\}]==0,}\\
\pmb{\text{If}[ (\text{i1}==\text{i2}==\text{i3}==\text{i4}==0 )\|( \text{j1}==\text{j2}==\text{j3}==\text{j4}==0), , }\\
\pmb{\text{valueisprime}\text{:=}\text{False} ]] ] ] ] ] ] ] ] ] ; \text{valueisprime}]}\)
\end{doublespace}

\begin{doublespace}
\noindent\(\pmb{\text{}}\\
\pmb{\text{isprime}}\)
\end{doublespace}

\begin{doublespace}
\noindent\(\text{True}\)
\end{doublespace}

\bibliographystyle{amsalpha}
{ \small
\bibliography{references}

\newcommand{\etalchar}[1]{$^{#1}$}
\providecommand{\bysame}{\leavevmode\hbox to3em{\hrulefill}\thinspace}
\providecommand{\MR}{\relax\ifhmode\unskip\space\fi MR }
\providecommand{\MRhref}[2]{%
  \href{http://www.ams.org/mathscinet-getitem?mr=#1}{#2}
}
\providecommand{\href}[2]{#2}
\begin{thebibliography}{BDGGH{\etalchar{+}}92}

\bibitem[AA18]{alvarado_auffarth}
Mat\'{\i}as Alvarado and Robert Auffarth, \emph{Fixed points of endomorphisms
  of complex tori}, J. Algebra \textbf{507} (2018), 428--438. \MR{3807055}

\bibitem[AC08]{amerik_campana}
Ekaterina Amerik and Fr\'{e}d\'{e}ric Campana, \emph{Fibrations m\'{e}romorphes
  sur certaines vari\'{e}t\'{e}s \`a fibr\'{e} canonique trivial}, Pure Appl.
  Math. Q. \textbf{4} (2008), no.~2, Special Issue: In honor of Fedor
  Bogomolov. Part 1, 509--545. \MR{2400885}

\bibitem[BC16]{blanc_cantat}
J{\'e}r{\'e}my Blanc and Serge Cantat, \emph{Dynamical degrees of birational
  transformations of projective surfaces}, J. Amer. Math. Soc. \textbf{29}
  (2016), no.~2, 415--471. \MR{3454379}

\bibitem[BD11]{baker_demarco}
Matthew Baker and Laura DeMarco, \emph{Preperiodic points and unlikely
  intersections}, Duke Math. J. \textbf{159} (2011), no.~1, 1--29. \MR{2817647}

\bibitem[BDGGH{\etalchar{+}}92]{bdgps}
M.-J. Bertin, A.~Decomps-Guilloux, M.~Grandet-Hugot, M.~Pathiaux-Delefosse, and
  J.-P. Schreiber, \emph{Pisot and {S}alem numbers}, Birkh\"{a}user Verlag,
  Basel, 1992, With a preface by David W. Boyd. \MR{1187044}

\bibitem[BH16]{bauer_herrig}
Thomas Bauer and Thorsten Herrig, \emph{Fixed points of endomorphisms on
  two-dimensional complex tori}, J. Algebra \textbf{458} (2016), 351--363.
  \MR{3500781}

\bibitem[BL94]{birkenhake_lange_FP}
Ch. Birkenhake and H.~Lange, \emph{Fixed-point free automorphisms of abelian
  varieties}, Geom. Dedicata \textbf{51} (1994), no.~3, 201--213. \MR{1293798}

\bibitem[BL04]{birkenhake_lange}
Christina Birkenhake and Herbert Lange, \emph{Complex abelian varieties},
  second ed., Grundlehren der Mathematischen Wissenschaften [Fundamental
  Principles of Mathematical Sciences], vol. 302, Springer-Verlag, Berlin,
  2004. \MR{2062673}

\bibitem[BS91]{bedford_smillie}
Eric Bedford and John Smillie, \emph{Polynomial diffeomorphisms of {${\bf
  C}^2$}. {II}. {S}table manifolds and recurrence}, J. Amer. Math. Soc.
  \textbf{4} (1991), no.~4, 657--679. \MR{1115786}

\bibitem[Dai06]{daileda}
Ryan~C. Daileda, \emph{Algebraic integers on the unit circle}, J. Number Theory
  \textbf{118} (2006), no.~2, 189--191. \MR{2223980}

\bibitem[DF01]{diller_favre}
Jeffrey Diller and Charles Favre, \emph{Dynamics of bimeromorphic maps of
  surfaces}, Amer. J. Math. \textbf{123} (2001), no.~6, 1135--1169.
  \MR{1867314}

\bibitem[DF17]{favre_dujardin_manin}
R.~Dujardin and C.~Favre, \emph{The dynamical {M}anin-{M}umford problem for
  plane polynomial automorphisms}, J. Eur. Math. Soc. (JEMS) \textbf{19}
  (2017), no.~11, 3421--3465. \MR{3713045}

\bibitem[DS05]{dinh_sibony_green_currents}
Tien-Cuong Dinh and Nessim Sibony, \emph{Green currents for holomorphic
  automorphisms of compact {K}\"{a}hler manifolds}, J. Amer. Math. Soc.
  \textbf{18} (2005), no.~2, 291--312. \MR{2137979}

\bibitem[DS10]{dinh_sibony_super}
\bysame, \emph{Super-potentials for currents on compact {K}\"{a}hler manifolds
  and dynamics of automorphisms}, J. Algebraic Geom. \textbf{19} (2010), no.~3,
  473--529. \MR{2629598}

\bibitem[DTV10]{dethelin_vigny}
Henry De~Th\'{e}lin and Gabriel Vigny, \emph{Entropy of meromorphic maps and
  dynamics of birational maps}, M\'{e}m. Soc. Math. Fr. (N.S.) (2010), no.~122,
  vi+98. \MR{2752759}

\bibitem[FG20]{favre2020arithmetic}
Charles Favre and Thomas Gauthier, \emph{The arithmetic of polynomial dynamical
  pairs}, 2020.

\bibitem[Ghy95]{ghys_95}
\'{E}tienne Ghys, \emph{Holomorphic {A}nosov systems}, Invent. Math.
  \textbf{119} (1995), no.~3, 585--614. \MR{1317651}

\bibitem[Gro87]{gromov}
M.~Gromov, \emph{Hyperbolic groups}, Essays in group theory, Math. Sci. Res.
  Inst. Publ., vol.~8, Springer, New York, 1987, pp.~75--263. \MR{919829}

\bibitem[Gro03]{gromov_03}
Misha Gromov, \emph{On the entropy of holomorphic maps}, L'Enseign Math.
  \textbf{(2) 49} (2003), no.~3-4, 217--235. \MR{2026895}

\bibitem[GS06]{gille_szamuely}
Philippe Gille and Tam\'{a}ls Szamuely, \emph{Central simple algebras and
  galois cohomology}, first ed., Cambridge Studies in Advanced Mathematics,
  vol. 101, Cambridge University Press, 2006. \MR{MR2266528}

\bibitem[GS17]{ghioca_scanlon}
Dragos Ghioca and Thomas Scanlon, \emph{Density of orbits of endomorphisms of
  abelian varieties}, Trans. Amer. Math. Soc. \textbf{369} (2017), no.~1,
  447--466. \MR{3557780}

\bibitem[GTZ11]{ghioca_tucker_zhang}
Dragos Ghioca, Thomas~J. Tucker, and Shouwu Zhang, \emph{Towards a dynamical
  {M}anin-{M}umford conjecture}, Int. Math. Res. Not. IMRN (2011), no.~22,
  5109--5122. \MR{2854724}

\bibitem[Gue10]{guedj_proprietes}
Vincent Guedj, \emph{Propri\'{e}t\'{e}s ergodiques des applications
  rationnelles}, Quelques aspects des syst\`emes dynamiques polynomiaux, Panor.
  Synth\`eses, vol.~30, Soc. Math. France, Paris, 2010, pp.~97--202.
  \MR{2932434}

\bibitem[GV94]{ghys_94}
\'{E}tienne Ghys and Alberto Verjovsky, \emph{Locally free holomorphic actions
  of the complex affine group}, Geometric study of foliations ({T}okyo, 1993),
  World Sci. Publ., River Edge, NJ, 1994, pp.~201--217. \MR{1363727}

\bibitem[Her19]{herrig_19}
Thorsten Herrig, \emph{Fixed points and entropy of endomorphisms on simple
  abelian varieties}, J. Algebra \textbf{517} (2019), 95--111. \MR{3869268}

\bibitem[HK03]{katok}
Boris Hasselblatt and Anatole Katok, \emph{A first course in dynamics},
  Cambridge University Press, New York, 2003, With a panorama of recent
  developments. \MR{1995704}

\bibitem[Jar14]{jarvis}
Frazer Jarvis, \emph{Algebraic number theory}, Springer Undergraduate
  Mathematics Series, Springer International Publishing, 2014. \MR{MR3289993}

\bibitem[KR17]{krieger_reschke}
Holly Krieger and Paul Reschke, \emph{Cohomological conditions on endomorphisms
  of projective varieties}, Bull. Soc. Math. France \textbf{145} (2017), no.~3,
  449--468. \MR{3766117}

\bibitem[Laz04]{lazarsfeld_positivity_1}
Robert Lazarsfeld, \emph{Positivity in algebraic geometry. {I}}, Ergebnisse der
  Mathematik und ihrer Grenzgebiete. 3. Folge. A Series of Modern Surveys in
  Mathematics [Results in Mathematics and Related Areas. 3rd Series. A Series
  of Modern Surveys in Mathematics], vol.~48, Springer-Verlag, Berlin, 2004,
  Classical setting: line bundles and linear series. \MR{2095471}

\bibitem[McM02]{McMullen02}
Curtis~T. McMullen, \emph{Dynamics on {$K3$} surfaces: {S}alem numbers and
  {S}iegel disks}, J. Reine Angew. Math. \textbf{545} (2002), 201--233.
  \MR{1896103}

\bibitem[McM11]{McMullen11}
\bysame, \emph{K3 surfaces, entropy and glue}, J. Reine Angew. Math.
  \textbf{658} (2011), 1--25. \MR{2831510}

\bibitem[MOR18]{MOR18}
Yuya Matsumoto, Hisanori Ohashi, and S\l~awomir Rams, \emph{On automorphisms of
  {E}nriques surfaces and their entropy}, Math. Nachr. \textbf{291} (2018),
  no.~13, 2084--2098. \MR{3858678}

\bibitem[Neu99]{neukirch}
Jürgen Neukirch, \emph{Algebraic number theory}, Grundlehren der
  Mathematischen Wissenschaften [Fundamental Principles of Mathematical
  Sciences], vol. 322, Springer-Verlag, Berlin, 1999. \MR{1697859}

\bibitem[Ogu10]{oguiso}
Keiji Oguiso, \emph{The third smallest {S}alem number in automorphisms of
  {$K3$} surfaces}, Algebraic geometry in {E}ast {A}sia---{S}eoul 2008, Adv.
  Stud. Pure Math., vol.~60, Math. Soc. Japan, Tokyo, 2010, pp.~331--360.
  \MR{2761934}

\bibitem[Ogu16]{oguiso_simple}
\bysame, \emph{Simple abelian varieties and primitive automorphisms of null
  entropy of surfaces}, K3 surfaces and their moduli, Progr. Math., vol. 315,
  Birkh\"{a}user/Springer, [Cham], 2016, pp.~279--296. \MR{3524172}

\bibitem[Ogu19]{oguiso_pisot}
\bysame, \emph{Pisot units, {S}alem numbers, and higher dimensional projective
  manifolds with primitive automorphisms of positive entropy}, Int. Math. Res.
  Not. IMRN (2019), no.~5, 1373--1400. \MR{3920351}

\bibitem[OT14]{truong_oguiso_salem}
Keiji Oguiso and Tuyen~Trung Truong, \emph{Salem numbers in dynamics on
  {K}\"{a}hler threefolds and complex tori}, Math. Z. \textbf{278} (2014),
  no.~1-2, 93--117. \MR{3267571}

\bibitem[OY20]{oguiso_yu_20}
Keiji Oguiso and Xun Yu, \emph{Minimum positive entropy of complex enriques
  surface automorphisms}, to appear in Duke Math. J (2020).

\bibitem[PR02]{pink_roessler}
Richard Pink and Damian Roessler, \emph{On {H}rushovski's proof of the
  {M}anin-{M}umford conjecture}, Proceedings of the {I}nternational {C}ongress
  of {M}athematicians, {V}ol. {I} ({B}eijing, 2002), Higher Ed. Press, Beijing,
  2002, pp.~539--546. \MR{1989204}

\bibitem[Ray83]{raynaud}
M.~Raynaud, \emph{Sous-vari\'{e}t\'{e}s d'une vari\'{e}t\'{e} ab\'{e}lienne et
  points de torsion}, Arithmetic and geometry, {V}ol. {I}, Progr. Math.,
  vol.~35, Birkh\"{a}user Boston, Boston, MA, 1983, pp.~327--352. \MR{717600}

\bibitem[Res12]{Reschke12}
Paul Reschke, \emph{Salem numbers and automorphisms of complex surfaces}, Math.
  Res. Lett. \textbf{19} (2012), no.~2, 475--482. \MR{2955777}

\bibitem[Res17]{Reschke17}
\bysame, \emph{Salem numbers and automorphisms of abelian surfaces}, Osaka J.
  Math. \textbf{54} (2017), no.~1, 1--15. \MR{3619745}

\bibitem[Sal63]{salem}
Rapha\"{e}l Salem, \emph{Algebraic numbers and {F}ourier analysis}, D. C. Heath
  and Co., Boston, Mass., 1963. \MR{0157941}

\bibitem[Vig80]{vigneras}
Marie-France Vign\'{e}ras, \emph{Arithm\'{e}tique des alg\`ebres de
  quaternions}, Lecture Notes in Mathematics, vol. 800, Springer, Berlin, 1980.
  \MR{580949}

\bibitem[Vig14]{vigny_maxi}
Gabriel Vigny, \emph{Hyperbolic measure of maximal entropy for generic rational
  maps of {$\Bbb P^k$}}, Ann. Inst. Fourier (Grenoble) \textbf{64} (2014),
  no.~2, 645--680. \MR{3330918}

\bibitem[Voi19]{voight}
John Voight, \emph{Quaternion algebras}, 2019.

\bibitem[Wal75]{walters}
Peter Walters, \emph{Ergodic theory---introductory lectures}, Lecture Notes in
  Mathematics, Vol. 458, Springer-Verlag, Berlin-New York, 1975. \MR{0480949}

\bibitem[Yom87]{yomdin_87}
Yosef Yomdin, \emph{Volume growth and entropy}, Israel Journal of Mathematics
  \textbf{57} (1987), no.~3, 285--300.

\bibitem[Zha06]{zhang}
Shou-Wu Zhang, \emph{Distributions in algebraic dynamics}, Surveys in
  differential geometry. {V}ol. {X}, Surv. Differ. Geom., vol.~10, Int. Press,
  Somerville, MA, 2006, pp.~381--430. \MR{2408228}

\end{thebibliography}

}

\bibliographystyle{amsalpha}

\footnotesize
   \bigskip
   Nguyen-Bac Dang,
   Institute for Mathematical Science,
   Stony Brook University,
   Stony Brook NY 11794-3660, USA.

   \nopagebreak
   \textit{E-mail address:} \texttt{nguyen-bac.dang@stonybrook.edu}

   \bigskip
   Thorsten Herrig,
   Institut f\"ur Mathematik,
   Humboldt-Universit\"at zu Berlin,
   Rudower Chaussee 25,
   D-12489 Berlin, Germany.

   \nopagebreak
   \textit{E-mail address:} \texttt{herrigth@math.hu-berlin.de}

\end{document}